\documentclass[12pt, twoside]{article}
\usepackage{amsmath,amsthm,amssymb}
\usepackage{times}
\usepackage{enumerate}

\usepackage{fancyhdr}
\usepackage{supertabular}
\usepackage{setspace}
\usepackage{color}
\usepackage{longtable}

\usepackage[matrix,arrow,curve]{xy}
\usepackage{graphicx}
\usepackage{enumerate}
\usepackage{amsfonts}
\usepackage{cite}

\def\titlerunning#1{\gdef\titrun{#1}}
\makeatletter
\def\author#1{\gdef\autrun{\def\and{\unskip, }#1}\gdef\@author{#1}}
\def\address#1{{\def\and{\\\hspace*{18pt}}\renewcommand{\thefootnote}{}%
\footnote {#1}}%
\markboth{\autrun}{\titrun}}
\makeatother
\def\email#1{e-mail: #1}
\def\subjclass#1{{\renewcommand{\thefootnote}{}%
\footnote{\emph{Mathematics Subject Classification (2010):} #1}}}
\def\keywords#1{\par\medskip
\noindent\textbf{Keywords.} #1}

\newtheorem{thm}{Theorem}[section]

\newtheorem{prob}[thm]{Problem}

\newtheorem{theorem}{Theorem}[section]
\newtheorem{corollary}[theorem]{Corollary}
\newtheorem{lemma}[theorem]{Lemma}

\newtheorem{mainthm}[thm]{Main Theorem}

\theoremstyle{definition}
\newtheorem{defin}[thm]{Definition}

\theoremstyle{definition}
\newtheorem{definition}[theorem]{Definition}
\newtheorem{remark}[theorem]{Remark}
\newtheorem{example}[theorem]{Example}
\newtheorem{notation}[theorem]{Notation}

\newtheorem*{xrem}{Remark}

\numberwithin{equation}{section}

\newcommand{\mult}{\operatorname{mult}}
\newcommand{\qsim}{\sim_{\mathbb{Q}}}
\newcommand{\red}{\textcolor{black}}

\begin{document}


\baselineskip=17pt


\titlerunning{Affine cones over smooth cubic surfaces}

\title{Affine cones over smooth cubic surfaces}

\author{Ivan Cheltsov
\and 
Jihun Park 
\and 
Joonyeong Won}

\date{}

\maketitle

\address{I. Cheltsov: 
School of Mathematics, The University of Edinburgh, James Clerk Maxwell Building, The King's Buildings, Mayfield Road, Edinburgh EH9 3JZ, UK;
\email{I.Cheltsov@ed.ac.uk}
\and
J. Park: 
Center for Geometry and Physics, Institute for Basic Science (IBS), 77 Cheongam-ro, Nam-gu, Pohang, Gyeongbuk, 790-784, Korea;
Department of Mathematics, POSTECH, 77 Cheongam-ro, Nam-gu, Pohang, Gyeongbuk, 790-784, Korea;
\email{wlog@postech.ac.kr}
\and
J. Won: 
KIAS 85 Hoegiro Dongdaemun-gu, Seoul 130-722, Korea;
\email{leonwon@kias.re.kr}}


\subjclass{Primary 14E15, 14J45, 14R20, 14R25; Secondary 14C20, 14E05, 14J17}


\begin{abstract}
We show that affine cones over smooth cubic surfaces do not admit
non-trivial $\mathbb{G}_{a}$-actions.

\keywords{affine cone, $\alpha$-invariant, anticanonical divisor, cylinder, del Pezzo surface, $\mathbb{G}_a$-action, log canonical singularity.}
\end{abstract}

Throughout this article, we assume that all considered varieties are
algebraic and defined over an algebraically closed field of characteristic $0$.

\section{Introduction}
\label{sec:intro}
One of the motivations for the present article originates from the articles of H.~A. Schwartz (\cite{Sch1873})  and G.~H. Halphen (\cite{Ha1883}) in the middle of 19th century, where they studied polynomial solutions of Brieskorn-Pham polynomial equations in three variables after L.~Euler (1756), J.~Liouville (1879) and so fourth (\cite{DaGra95}). Meanwhile, since the middle of 20th century  the study of rational singularities has witnessed great development (\cite{Ar66}, \cite{Br68}, \cite{Lau72}). These two topics, one classic and the other modern, encounter each other in contemporary mathematics. For instance, there is a strong connection between the existence of a rational curve on a normal affine surface, i.e., a polynomial solution to algebraic equations,  and rational singularities (\cite{FZ2003}). 

As an additive analogue of toric geometry, unipotent group actions, specially $\mathbb{G}_a$-actions, on varieties are very attractive objects to study. Indeed, $\mathbb{G}_a$-actions have been investigated for their own sake (\cite{Ba84}, \cite{Ka04}, \cite{MaMi03}, \cite{Sn89}, \cite{Wi90}). 
We also observe that $\mathbb{G}_a$-actions appear in the study of rational singularities. In particular, the article \cite{FZ2003} shows that a Brieskorn-Pham surface singularity is a cyclic quotient singularity if and only if the surface admits a non-trivial  
regular $\mathbb{G}_{a}$-action. 
Considering its 3-dimensional analogue,  H.~Flenner and M.~Zaidenberg in 2003  proposed the following question (\cite[Question~2.22]{FZ2003}):\medskip
\begin{center}
\begin{minipage}{.8\linewidth}
 \emph{Does the affine Fermat cubic threefold
$
x^3+y^3+z^3+w^3=0$
in $\mathbb{A}^4$ admit a non-trivial regular $\mathbb{G}_{a}$-action?
}
\end{minipage}
\end{center}
\medskip
Even though it is simple-looking, this problem stands open for 10 years.
It turns out that this problem is purely
geometric and can be considered in a much wider setting 
(\cite{KPZ11a}, \cite{KPZ11b}, \cite{KPZ12a}, \cite{KPZ12b},~\cite{Pe11}).

To see the problem from a wider view point, we let $X$ be a smooth projective variety with a polarisation $H$, where $H$ is an ample
divisor on $X$. The \emph{generalized cone} over
$(X,H)$ is the  affine variety defined by
$$
\hat{X}=\mathrm{Spec}\left(\bigoplus_{n\geqslant 0}H^0\left(X, \mathcal{O}_{X}\left(nH\right)\right)\right).%
$$
\begin{remark}
\label{remark:generalized-cone} The affine variety $\hat{X}$ is the usual cone
over $X$ embedded in a projective space by the linear system $|H|$
provided that $H$ is very ample and the image of the variety $X$
is projectively normal.
\end{remark}

Let $S_d$ be a smooth del Pezzo surface of degree $d$ and let $\hat{S}_d$ be the generalized cone  over $(S_d, -K_{S_d})$.  
For $3\leqslant d\leqslant 9$, the anticanonical divisor $-K_{S_d}$ is very ample and the generalized cone $\hat{S}_d$ is the affine cone in $\mathbb{A}^{d+1}$ over  the smooth variety anticanonically embedded in $\mathbb{P}^d$. In particular, for $d=3$, the cubic surface $S_3$ is defined by a cubic homogenous polynomial equation $F(x,y,z,w)=0$ in $\mathbb{P}^3$, and hence the generalized cone $\hat{S}_3$ is the affine hypersurface in $\mathbb{A}^4$ defined by the equation $F(x,y,z,w)=0$.
For $d=2$, the generalized cone $\hat{S}_2$ is the affine cone in $\mathbb{A}^{4}$ over  the smooth hypersurface  in the weighted projective space
$\mathbb{P}(1,1,1,2)$ defined by a quasi-homogeneous polynomial of degree $4$. For $d=1$, the generalized cone $\hat{S}_1$ is the affine cone in $\mathbb{A}^{4}$ over  the smooth hypersurface  in the weighted projective space
$\mathbb{P}(1,1,2,3)$ defined by a quasi-homogeneous polynomial of degree $6$ (\cite[Theorem~4.4]{HiWa81}).

It is natural to ask
whether the affine variety $\hat{S}_d$ admits a non-trivial $\mathbb{G}_a$-action.
The problem at the beginning is just  a special case of this.

T.~Kishimoto, Yu.~Prokhorov and M.~Zaidenberg have been studying this generalized problem and proved the following:
\begin{theorem}\label{theorem:KPZ-actions-4-9}
If $4\leqslant d\leqslant 9$, then the generalized cone $\hat{S}_d$ admits an effective $\mathbb{G}_a$-action.
\end{theorem}
\begin{proof}
See \cite[Theorem~3.19]{KPZ11a}.
\end{proof}
\begin{theorem}\label{theorem:KPZ-actions-1-2}
If $d\leqslant 2$, then the generalized cone $\hat{S}_d$ does not admit a non-trivial $\mathbb{G}_a$-action.
\end{theorem}
\begin{proof}
See \cite[Theorem~1.1]{KPZ12b}.
\end{proof}
Their proofs make good use of a geometric property called cylindricity, which is worthwhile to be studied for its own sake.
\begin{definition}[\cite{KPZ11a}]
\label{definition:polar-cylinder} Let $M$  be a $\mathbb{Q}$-divisor on  a  smooth projective variety $X$.  An $M$-polar cylinder in $X$ is
an open subset
$$
U=X\setminus \mathrm{Supp}(D)
$$
defined by an effective
$\mathbb{Q}$-divisor $D$ on $X$ with $D\sim_{\mathbb{Q}} M$  such that $U$ is isomorphic to $Z\times\mathbb{A}^1$ for some affine variety $Z$.
\end{definition}
They show that the existence of an $H$-polar cylinder on $X$ is equivalent to the existence of a non-trivial $\mathbb{G}_a$-action on the generalized cone over $(X, H)$.

\begin{lemma}\label{theorem:KPZ-criterion} Let $H$ be an ample Cartier divisor on a smooth projective  variety $X$. Suppose that the generalized  cone $\hat{X}$ over $(X, H)$ is normal. Then the generalized cone $\hat{X}$
admits an effective $\mathbb{G}_{a}$-action if and only if $X$
contains an $H$-polar cylinder.
\end{lemma}
\begin{proof}
See \cite[Corollary~3.2]{KPZ12a}.
\end{proof}
\begin{remark}
If $X$ is a rational surface, then there always exists an ample
Cartier divisor $H$ on $X$ such that $\hat{X}$ is normal and $X$
contains an $H$-polar cylinder (see
\cite[Proposition~3.13]{KPZ11a}), which implies, in particular,
that $\hat{X}$ admits an effective $\mathbb{G}_{a}$-action.
\end{remark}
Indeed, what T.~Kishimoto, Yu.~Prokhorov and M.~Zaidenberg proved for their two theorems is that the del Pezzo surface $S_d$ has 
a $(-K_{S_d})$-polar cylinder if $4\leqslant d\leqslant 9$ but  no  $(-K_{S_d})$-polar cylinder if $d\leqslant 2$.

The main result of the present article is

\begin{theorem}
\label{theorem:main} A smooth cubic surface $S_3$ in $\mathbb{P}^3$  does
not contain any $(-K_{S_3})$-polar cylinders.
\end{theorem}
Together with Theorems~\ref{theorem:KPZ-actions-4-9} and~\ref{theorem:KPZ-actions-1-2}, this makes us reach  the following conclusion via Lemma~\ref{theorem:KPZ-criterion}.
\begin{corollary}
\label{corollary:main} Let $S_d$ be a smooth del Pezzo surface of degree $d$.  Then \red{the generalized cone over $(S_d,-K_{S_d})$} admits a non-trivial regular
$\mathbb{G}_{a}$-action if and only if $d\geqslant 4$.
\end{corollary}

In particular, we here present a long-expected answer to the question raised by H.~Flenner and M.~Zaidenberg.
\begin{corollary}
\label{corollary:FZ} The affine Fermat cubic threefold
$
x^3+y^3+z^3+w^3=0$
in $\mathbb{A}^4$ does not admit a non-trivial regular $\mathbb{G}_{a}$-action.\end{corollary}

The following lemma shows that having  anticanonical cylinders on del Pezzo surfaces is strongly related to the log canonical thresholds of their effective anticanonical $\mathbb{Q}$-divisors\footnote{An anticanonical $\mathbb{Q}$-divisor on a variety $X$ is a  $\mathbb{Q}$-divisor $\mathbb{Q}$-linearly equivalent to an anticanonical divisor of $X$, meanwhile, an effective  anticanonical divisor on $X$  is a member of the anticanonical linear system $|-K_X|$. }. It may also be one example that 
shows how important it is to study
 singularities of effective anticanonical $\mathbb{Q}$-divisors on  Fano manifolds.  Indeed, the proof of 
 Theorem~\ref{theorem:main} is substantially based on the lemma below.

\begin{lemma}
\label{lemma:KPZ-cubic} Let $S_d$ be a smooth del Pezzo surface of degree  $d\leqslant 4$. Suppose that $S_d$ contains a
$(-K_{S_d})$-polar cylinder, i.e., there is an open affine subset $U\subset S_d$ and an effective anticanonical $\mathbb{Q}$-divisor $D$ such that $U=S_d\setminus \mathrm{Supp}(D)$ and $U\cong Z\times\mathbb{A}^1$ for some smooth rational affine curve $Z$.
Then there exists a point $P$ on $S_d$ such that
\begin{itemize}
\item the log pair $(S_d, D)$ is not log canonical at the point $P$;
\item if there exists a unique divisor $T$ in the anticanonical  linear system $|-K_{S_d}|$ such that the log pair $(S_d, T)$ is not log canonical at  $P$, then there is an effective anticanonical $\mathbb{Q}$-divisor $D'$ on the surface $S_d$ such that 
\begin{itemize}
\item the log pair $(S_d, D')$ is not log canonical at $P$;
\item the support of the divisor $T$ is not contained in the support of $D'$.
\end{itemize}
\end{itemize}
\end{lemma}

\begin{proof}
\red{This follows from \cite[Lemma~4.11]{KPZ11a} and the proof of
\cite[Lemma~4.14]{KPZ11a} (cf. the proof of
\cite[Lemma~5.3]{KPZ12b}). Since the proof is presented dispersedly in \cite{KPZ11a} or \cite{KPZ12b}, for the convenience of the readers, we
give a detailed and streamlined proof in Appendix~\ref{sec:KPZ}.}
\end{proof}

Applying Lemma~\ref{lemma:convexity} \red{below}, we easily obtain

\begin{corollary}
\label{corollary:KPZ-cubic}  Let $S_3$ be a smooth del Pezzo surface of degree  $3$. Suppose that $S_3$ contains a
$(-K_{S_3})$-polar cylinder. Then there is an effective anticanonical 
$\mathbb{Q}$-divisor $D$ on $S_3$ such that
\begin{itemize}
\item the log pair $(S_3,D)$ is not log
canonical at some point $P$ on  $S_3$;
\item the support of $D$ does not
contain at least one irreducible component of the tangent hyperplane section $T_P$ of $S_3$ at the point $P$.
\end{itemize}
\end{corollary}

In order to prove Theorem~\ref{theorem:main}, it suffices to show that there is no such a divisor $D$ described in 
Corollary~\ref{corollary:KPZ-cubic} on a smooth del Pezzo surface of degree $3$. 
In this article, this will be done in a bit wider setting. To be precise, we prove 

\begin{theorem}
\label{theorem:technical} Let $S_d$ be a smooth del Pezzo surface of degree $d\leqslant 3$ and let
$D$ be an effective anticanonical $\mathbb{Q}$-divisor on $S_d$. Suppose that
the log pair $(S_d,D)$ is not log
canonical at a point $P$.
Then there exists a unique divisor $T$ in the anticanonical linear system
$|-K_{S_d}|$ such that the log pair $(S_d,T)$ is not log canonical at the point $P$. Moreover,
the support of $D$ contains all the irreducible components of
$\mathrm{Supp}(T)$.
\end{theorem}
\begin{corollary}
\label{corollary:technical} Let $S_3$ be  a smooth cubic surface in
$\mathbb{P}^3$ and let
$D$ be an effective anticanonical $\mathbb{Q}$-divisor on $S_3$. Suppose  that
 the log pair $(S_3,D)$ is not log
canonical at a point $P$.
Then for the tangent hyperplane section $T_P$ at the point $P$, the log pair
$(S_3,T_P)$ is not log canonical at $P$
and $\mathrm{Supp}(D)$ contains all the irreducible components of
$\mathrm{Supp}(T_P)$.
\end{corollary}

Note that Corollary~\ref{corollary:technical} contradicts 
the conclusion of Corollary~\ref{corollary:KPZ-cubic}. \red{This} simply
means that the hypothesis of Corollary~\ref{corollary:KPZ-cubic}
fails to be true. This shows that Theorem~\ref{theorem:technical}
implies Theorem~\ref{theorem:main}. Moreover, we see that
Theorem~\ref{theorem:technical} recovers 
Theorem~\ref{theorem:KPZ-actions-1-2} through Lemma~\ref{lemma:KPZ-cubic} as well.

\begin{remark}
\label{remark:del-Pezzo-high-degree}  The condition $d \leqslant 3$ is 
crucial in Theorem~\ref{theorem:technical}. Indeed,
if $d\geqslant 4$, then  the assertion of
Theorem~\ref{theorem:technical} is no longer true (\red{see}
the proof of \cite[Theorem~3.19]{KPZ11a}). For example, consider the case when $d=4$.
There exists a birational
morphism $f\colon S_4\to\mathbb{P}^2$ such that $f$ is the blow up of $\mathbb{P}^2$  at
five points that lie on a unique irreducible conic. Denote this conic by $C$.
Let $\tilde{C}$ be the proper transform of the conic $C$ on the
surface $S_4$ and let $E_1, \ldots, E_5$ be the exceptional
divisors of the morphism $f$. Put
$$
D=\frac{3}{2}\tilde{C}+\sum_{i=1}^{5}\frac{1}{2}E_i.
$$
It is an effective anticanonical  $\mathbb{Q}$-divisor on $S_4$ and the log pair $(S_4,D)$ is not log canonical 
at  any point $P$ on $\tilde{C}$.
Moreover, for any $T\in |-K_{S_4}|$,  its support cannot be contained in the support of the divisor
$D$.
\end{remark}

To our surprise, Theorem~\ref{theorem:technical} has other
applications that are interesting for their own sake. 

From here to the end of this section, let $X$ be a projective variety with at worst Kawamata log terminal singularities and let $H$ be an ample divisor on $X$.

\begin{definition} The $\alpha$-invariant of the log pair $(X,H)$ is the number defined by 
$$
\alpha\left(X, H\right)=\mathrm{sup}\left\{\lambda\in\mathbb{Q}\ \left|%
\aligned
&\text{the log pair}\ \left(X, \lambda D\right)\ \text{is log canonical for every} \\
&\text{effective $\mathbb{Q}$-divisor}\ D\ \text{on}\ X\ \text{with}\ D\sim_{\mathbb{Q}} H.\\
\endaligned\right.\right\}.%
$$
\end{definition}

The invariant $\alpha(X,H)$ has been studied intensively by many
\red{authors} who used different notations for $\alpha(X,H)$ (\cite{Am06},
\cite{Ch01b}, \cite{dFEM03}, \cite[\S~3.4]{Berman}
\cite[Definition~3.1.1]{CheltsovParkWon},
\cite[Appendix~A]{ChSh08c}, \cite[Appendix~2]{Tian2008}). The notation $\alpha(X,H)$ is due to G.~Tian
who defined $\alpha(X,H)$ in a different way  (
\cite[Appendix~2]{Tian2008}). However, both the definitions coincide by
\cite[Theorem~A.3]{ChSh08c}. In the case when $X$ is a Fano
variety, the invariant $\alpha(X, -K_X)$ is known as
the famous $\alpha$-invariant of Tian and it is denoted simply by
$\alpha(X)$. The $\alpha$-invariant of Tian plays a very important
role in K\"ahler geometry due to the following.

\begin{theorem}[{\cite{DeKo01}, \cite{Na90}, \cite{Tian}}]
\label{theorem:alpha} Let $X$ be a Fano variety of dimension $n$ with at worst quotient singularities.
 If
$\alpha(X)>\frac{n}{n+1}$, then $X$ admits an orbifold
K\"ahler--Einstein metric.
\end{theorem}

The exact values of the $\alpha$-invariants of smooth del Pezzo surfaces, as below, have been obtained in \cite[Theorem~1.7]{Ch07a}. Those of del Pezzo surfaces defined over a field of positive characteristic are presented in 
\cite[Theorem~1.6]{Jesus} and those of del Pezzo surfaces with du Val singularities in \cite{ChK10} and \cite{PaW10}.
\begin{theorem}
\label{theorem:GAFA} Let $S_d$ be a smooth del Pezzo surface of degree $d$. Then
\[\aligned
&\alpha(S_d)=\left\{%
\aligned
&1/3\ \mathrm{if}\ d=9, 7 \mbox{ or } \red{d=8 \mbox{ and }} S_8=\mathbb{F}_1;\\%
&1/2\ \mathrm{if}\ d=6, 5 \mbox{ or }  \red{d=8 \mbox{ and }} S_8=\mathbb{P}^1\times\mathbb{P}^1;\\%
&2/3\ \mathrm{if}\  d=4;\\%
\endaligned\right.%
\\
&\alpha(S_3)=\left\{%
\aligned
&2/3\ \mathrm{if}\ S_3\ \mathrm{is\ a\ cubic\ surface\ in}\ \mathbb{P}^{3}\ \mathrm{with\ an\ Eckardt\ point};\\%
&3/4\ \mathrm{if}\ S_3\ \mathrm{is\ a\ cubic\ surface\ in}\ \mathbb{P}^{3}\ \mathrm{without\ Eckardt\ points};\\%
\endaligned\right.%
\\
&
\alpha(S_2)=\left\{%
\aligned
&3/4\ \mathrm{if}\  |-K_{S_2}|\ \mathrm{has\ a\ tacnodal\ curve};\\%
&5/6\ \mathrm{if}\   |-K_{S_2}|\ \mathrm{has\ no\ tacnodal\ curves};\\%
\endaligned\right. %
\\
&\alpha(S_1)=\left\{%
\aligned
&5/6\ \mathrm{if}\  |-K_{S_1}|\ \mathrm{has\ a\ cuspidal\ curve};\\%
&1\ \mathrm{if}\  |-K_{S_1}|\ \mathrm{has\ no\ cuspidal\ curves}.\\%
\endaligned\right.\\
\endaligned
\]
\end{theorem}

\begin{remark}\label{remark:alpha-dp1-dp2-dp3} Theorem~\ref{theorem:technical} also provides the exact values of the $\alpha$-invariants for smooth del Pezzo surfaces of degrees $\leqslant 3$. We here show how to extract the values from Theorem~\ref{theorem:technical}. 
Let $\nu$ be the greatest number such that $(S_d, \nu C)$ is log canonical for every member $C$ in $|-K_{S_d}|$. The number $\nu$ can be easily obtained from   \cite[Section ~3]{Pa01} and checked to be the same as the value listed in Theorem~\ref{theorem:GAFA} for the $\alpha$-invariant of $S_d$.
By the definition of $\nu$, there is an effective anticanonical divisor $C$ on the surface $S_d$ such
that $(S_d,\nu C)$ is log canonical but not Kawamata log terminal. This gives us
$\alpha(S_d)\leqslant\nu$.

Suppose that $\alpha(S_d)<\nu$. Then there are an effective anticanonical
$\mathbb{Q}$-divisor $D$ and a
positive rational number $\lambda<\nu$ such that $(S_d,\lambda D)$
is not log canonical at some point~$P$ on $S_d$. Since $\lambda<1$,
the log pair $(S_d,D)$ is not log canonical at the point $P$ either.
By Theorem~\ref{theorem:technical}, there exists a divisor $T\in
|-K_{S_d}|$ such that $(S_d,T)$ is not log canonical at $P$. In addition, 
$\mathrm{Supp}(D)$ contains all the irreducible components of
$\mathrm{Supp}(T)$.

The log pair $(S_d,\lambda T)$ is log canonical since $\lambda<\nu$.
Put $D_{\epsilon}=(1+\epsilon)D-\epsilon T$ for every non-negative
rational number $\epsilon$. Then $D_0=D$ and $D_\epsilon$ is
effective for $0<\epsilon\ll 1$ because $\mathrm{Supp}(D)$
contains all the irreducible components of $\mathrm{Supp}(T)$. Choose
the greatest $\epsilon$ such that $D_{\epsilon}$ is still
effective. Then $\mathrm{Supp}(D_{\epsilon})$ does not contain at
least one irreducible component of $\mathrm{Supp}(T)$.

Since $(S_d,\lambda T)$ is log canonical at $P$ and $(S_d,\lambda D)$
is not log canonical at $P$, the log pair $(S_d,\lambda
D_{\epsilon})$ is not log canonical at $P$ (see
Lemma~\ref{lemma:convexity}). In particular, the log pair $(S_d,
D_{\epsilon})$ is not log canonical at $P$. However, this contradicts  Theorem~\ref{theorem:technical} since  $D_{\epsilon}$ is an effective  anticanonical $\mathbb{Q}$-divisor. Therefore, $\alpha(S_d)=\nu$.
\end{remark}

\begin{corollary}
\label{corollary:alpha-dp1-dp2-dp3} Let $S_d$ be a smooth del Pezzo
surface of degree $d\leqslant 3$. If $d=3$, suppose,
in addition, that $S_3$ does not contain an Eckardt point. Then $S_d$
admits a K\"ahler--Einstein metric.
\end{corollary}

The problem on the existence of  K\"ahler--Einstein metrics on smooth
del Pezzo surfaces is completely solved by  G.~Tian and S.-T.~Yau in \cite{Ti90} and
\cite{TiYau87}. In particular,
Corollary~\ref{corollary:alpha-dp1-dp2-dp3} follows from
\cite[Main~Theorem]{Ti90}.

The invariant $\alpha(X,H)$ has a global nature. It measures the
singularities of effective $\mathbb{Q}$-divisors on $X$ in a fixed
$\mathbb{Q}$-linear equivalence class. F.~ Ambro suggested in
\cite{Am06}    a function
that encodes the local behavior of $\alpha(X,H)$.

\begin{definition}[{\cite{Am06}}]
\label{definition:alpha-p} The $\alpha$-function
$\alpha_X^H$  of the log pair $(X, H)$ is a function on $X$  into real numbers defined as follows: for a given point $P\in X$,
$$
\alpha_X^H(P)=\mathrm{sup}\left\{\lambda\in\mathbb{Q}\ \left|%
\aligned
&\text{the log pair}\ \left(X, \lambda D\right)\ \text{is log canonical at $P$ for}\\
&\text{every effective $\mathbb{Q}$-divisor}\ D\ \text{on}\ X\ \text{with}\ D\sim_{\mathbb{Q}} H.\\
\endaligned\right.\right\}.%
$$
\end{definition}

\begin{lemma}
\label{lemma:inf-sup} The identity $\alpha(X,H)=\inf_{P\in
X}\alpha_X^H(P)$ holds.
\end{lemma}

\begin{proof}
It is easy to check.
%
\end{proof}

In the case when $X$ is a Fano variety, we denote
the $\alpha$-function  of the log pair $(X,-K_X)$ simply by
$\alpha_X$.
\begin{example}
\label{example:alpha-P^n} One can easily see that
$\alpha_{\mathbb{P}^n}(P)\leqslant\frac{1}{n+1}$ for every
point $P$ on $\mathbb{P}^n$. This implies that the $\alpha$-function
$\alpha_{\mathbb{P}^n}$ is the constant function with the
value $\frac{1}{n+1}$ since $\alpha(\mathbb{P}^n)=\frac{1}{n+1}$.
\end{example}

\begin{example}
\label{example:alpha-product} It is easy to see
$\alpha_{\mathbb{P}^1\times\mathbb{P}^1}(P)\leqslant\frac{1}{2}$
for every point $P$ on $\mathbb{P}^1\times\mathbb{P}^1$. Since
$\alpha(\mathbb{P}^1\times\mathbb{P}^1)=\frac{1}{2}$ by
Theorem~\ref{theorem:GAFA}, the
$\alpha$-function $\alpha_{\mathbb{P}^1\times\mathbb{P}^1}$ is the constant
function with the value $\frac{1}{2}$ by
Lemma~\ref{lemma:inf-sup}. Moreover, if $X$ is a Fano variety with
at most Kawamata log terminal singularities, then the proof of
\cite[Lemma~2.21]{ChSh08c} shows that
$$
\alpha_{X\times\mathbb{P}^1}(P)=\mathrm{min}\left\{\frac{1}{2},\ \alpha_X\left(\mathrm{pr}_1\left(P\right)\right)\right\}
$$
for every point $P$ on $X\times \mathbb{P}^1$, where $\mathrm{pr}_1:X\times\mathbb{P}^1 \to X$ is the projection on the first factor. 
Using \red{the same argument as that} in the proof
of \cite[Lemma~2.29]{ChSh08c}, one can show that the $\alpha$-function of a
product of Fano varieties with at most  Gorenstein canonical 
singularities is the point-wise minimum of the pull-backs of the
$\alpha$-functions  on the factors.
\end{example}

As 
shown in Remark~\ref{remark:alpha-dp1-dp2-dp3}, the following can be obtained from  
Theorem~\ref{theorem:technical} in a similar manner.

\begin{corollary}
\label{corollary:Park-Won-dP1-dP2-dP3}Let $S_d$ be a smooth del Pezzo surface of degree $d\leqslant 3$. Then the $\alpha$-function of $S_d$ is as follows:
\[\aligned
&\alpha_{S_3}(P)=\left\{%
\aligned
&2/3\ \ \ \text{if the point $P$ is an Eckardt point;}\\
&3/4\ \ \ \text{if the tangent hyperplane section at $P$ has a tacnode at  $P$;}\\
&5/6\ \ \ \text{if the tangent hyperplane section at $P$ has a cusp at  $P$;}\\
&1 \ \ \ \text{otherwise;}\\
\endaligned\right.%
\\
&
\alpha_{S_2}(P)=\left\{%
\aligned
&3/4\ \ \ \text{if there is an effective anticanonical divisor with a tacnode at  $P$;}\\
&5/6\ \ \ \text{if there is an effective anticanonical divisor with a cusp at $P$;}\\
&1 \ \ \ \text{otherwise;}\\
\endaligned\right. %
\\
&\alpha_{S_1}(P)=\left\{%
\aligned
&5/6\ \ \ \text{if there is an effective anticanonical divisor with a cusp at $P$;}\\
&1 \ \ \ \text{otherwise.}\\
\endaligned\right.\\
\endaligned
\]
\end{corollary}

By Lemma~\ref{lemma:inf-sup},
Corollary~\ref{corollary:Park-Won-dP1-dP2-dP3} implies that 
Theorem~\ref{theorem:GAFA} holds for smooth del Pezzo surfaces of degrees at most $3$. Thus, it is quite
natural that we should extend 
Corollary~\ref{corollary:Park-Won-dP1-dP2-dP3} to all smooth del
Pezzo surfaces in order to obtain a functional generalisation of
Theorem~\ref{theorem:GAFA}. This will be done in
Section~\ref{sec:dP-alpha}, where we prove

\begin{theorem}
\label{theorem:Park-Won} 
Let $S_d$ be a smooth del Pezzo surface of degree $d\geqslant 4$. Then the $\alpha$-function of $S_d$ is as follows:
\[\aligned
&\alpha_{\mathbb{P}^2}(P)=1/3;\\
&\alpha_{\mathbb{F}_1}(P)=1/3; \ \ \ \ \ \alpha_{\mathbb{P}^1\times\mathbb{P}^1}(P)=1/2;\\
&\alpha_{S_7}(P)=\left\{
\aligned
& 1/3\ \ \ \aligned & \aligned & \text{if the point $P$ lies on the $(-1)$-curve that}\\
& \text{intersects two other $(-1)$-curves;}\\ \endaligned
 \endaligned\\
&1/2 \ \ \ \text{otherwise;}\\
\endaligned\right.\\
&\alpha_{S_6}(P)=1/2;\\
&\alpha_{S_5}(P)=\left\{
\aligned
&1/2 \ \ \ \text{if there is $(-1)$-curve passing through the point $P$;}\\
&2/3\ \ \ \text{if there is no $(-1)$-curve passing though $P$;}\\
\endaligned\right.
\\
&\alpha_{S_4}(P)=\left\{
\aligned &2/3\ \ \ \aligned & \text{if $P$ is on a $(-1)$-curve;}
 \endaligned\\
& 3/4\ \ \ \aligned & \aligned & \text{if there is an
effective anticanonical divisor that }\\
& \text{consists of two $0$-curves intersecting tangentially at  $P$;}\\ \endaligned
 \endaligned\\
&5/6\ \ \ \text{otherwise.}\\
\endaligned\right.\\
\endaligned
\]
\end{theorem}

\bigskip

The primary statement in this article is Theorem~\ref{theorem:technical}. 
As explained before, it immediately implies the main result of the article, Theorem~\ref{theorem:main} and also recovers Theorem~\ref{theorem:KPZ-actions-1-2}.
Theorem~\ref{theorem:technical} will be proved in the following way.

In
Section~\ref{section:preliminaries}, we review the results that will
be used in this article.  As a warm-up, we verify 
Theorem~\ref{theorem:technical} for a smooth
del Pezzo surface of degree $1$ (see Lemma~\ref{lemma:dP1}). This is very easy and instructive.

 In
Section~\ref{sec:dP2}, we establish two results about \emph{singular} del
Pezzo surfaces  of degree 2 that  play a role in the proof of
Theorems~\ref{theorem:technical} for smooth cubic surfaces. In addition, these two results immediately  yield   Theorem~\ref{theorem:technical} for
a smooth del Pezzo surface of degree $2$ (see
Lemma~\ref{lemma:dP2}).

In Section~\ref{sec:cubic-surfaces}, we
prove Theorem~\ref{theorem:technical}  for a smooth cubic surface.  This will be done by a thorough case-by-case analysis of all possible types of tangent hyperplane sections on a smooth cubic surface. Indeed, for a given point $P$ on the smooth cubic surface, we show that every effective anticanonical $\mathbb{Q}$-divisor is log canonical at $P$ if the tangent hyperplane section at $P$ is log canonical at $P$ 
(Lemmas~\ref{lemma:cubic-ODP},~\ref{lemma:cubic-triangle} and~\ref{lemma:cubic-nodal}), whereas we show that 
its support  contains the support of the tangent hyperplane section at $P$ if  an effective anticanonical $\mathbb{Q}$ divisor  and the tangent hyperplane section at $P$ are not log canonical at~$P$ (see Lemmas~\ref{lemma:cubic-Eckard},~\ref{lemma:cubic-tacnodal} and~\ref{lemma:cubic-cusp}).

The proof of Lemma~\ref{lemma:cubic-triangle} deserves a separate section because it is  the central and the most beautiful
part of the article and it is a bit lengthy. This   will be presented in Section~\ref{sec:three-lines}.

Appendix~\ref{sec:KPZ} will deal with Lemma~\ref{lemma:KPZ-cubic} for the readers' convenience.

\section{Preliminaries}%
\label{section:preliminaries}

This section presents simple but essential tools for the article.
Most of the described results here are well-known and  valid in much more general
settings (cf. \cite{Ko97}, \cite{KoMo} and \cite{La04II}).

Let $S$ be a~projective surface with at most du Val singularities, let $P$ be a smooth point of the surface $S$ and
let $D$ be an effective $\mathbb{Q}$-divisor on $S$.

\begin{lemma}
\label{lemma:mult} If the log pair $(S,D)$ is not log canonical at the smooth point $P$, then
$$\mathrm{mult}_{P}(D)>1.$$
\end{lemma}
\begin{proof}
This is a well-known fact. See \cite[Proposition~9.5.13]{La04II}, for instance.
\end{proof}

Write
$D=\sum_{i=1}^{r}a_{i}D_{i}$, where $D_{i}$'s are distinct prime divisors on the surface $S$ and $a_{i}$'s are positive rational numbers. 
\begin{lemma}
\label{lemma:convexity} Let $T$ be an effective
$\mathbb{Q}$-divisor on $S$  such that 
\begin{itemize}
\item $T\sim_{\mathbb{Q}} D$ but $T\ne D$;
 \item $T=\sum_{i=1}^{r}b_{i}D_{i}$ for some non-negative
rational numbers $b_1, \ldots, b_r$.
\end{itemize}  For every non-negative
rational number $\epsilon$, put
$D_\epsilon=(1+\epsilon)D-\epsilon T$. Then 
\begin{enumerate}
\item $D_{\epsilon}\sim_{\mathbb{Q}} D$
for every $\epsilon\geqslant 0$;
\item the set $
\left\{\epsilon\in\mathbb{Q}_{> 0}\ \vert\ D_\epsilon\ \text{is
effective}\right\}
$ attains the maximum $\mu$;
\item  the support of the divisor
$D_\mu$ does not contain at least one component of $\mathrm{Supp}(T)$;
\item if $(S,T)$ is log canonical at $P$ but $(S,D)$ is not
log canonical at $P$, then $(S,D_{\mu})$ is not log canonical at
$P$.
\end{enumerate}
\end{lemma}

\begin{proof}
The first assertion is obvious. 
For the rest we put
$$
c=\mathrm{max}\left\{\frac{b_{i}}{a_{i}}\ \Big\vert\ i=1, \ldots, r \right\}.%
$$
For some index $k$ we have
$c=\frac{b_{k}}{a_{k}}$. 

Suppose that $c\leqslant 1$. Then
$a_i\geqslant b_i$  for every $i$. This means that the
divisor $D-T=\sum_{i=1}^{r}(a_i-b_i)D_i$ is effective. However, it is impossible since $D-T$ is non-zero and numerically trivial on a
projective surface. Thus, $c>1$, and hence $b_k>a_k$.

Put $\mu=\frac{1}{c-1}$. Then
$\mu=\frac{a_k}{b_{k}-a_k}>0$ and 
$$
D_\mu=\frac{b_k}{b_{k}-a_k}D-\frac{a_k}{b_{k}-a_k}T=\sum_{i=1}^r\frac{b_ka_i-a_kb_i}{b_{k}-a_k}D_i,
$$
where $b_ka_i-a_kb_i\geqslant 0$ by the choice of 
$k$. In particular, the divisor $D_\mu$ is effective and its
support does not contain the curve $D_k$. Moreover, for every
positive rational number $\epsilon$, 
$D_{\epsilon}=\sum_{i=1}^{r}(a_i+\epsilon a_i-\epsilon b_i)D_i$.
If $\epsilon>\mu$, then
$$
\epsilon(b_k-a_k)>\mu(b_k-a_k)=\frac{a_k}{b_{k}-a_k}(b_k-a_k)=a_k,%
$$
and hence $D_\epsilon$ is not effective. This proves the second and the third assertions.

If both $(S,T)$  and $(S,D_{\mu})$ are log canonical at $P$, then $(S,D)$ must be log canonical at $P$
because $D=\frac{\mu}{1+\mu}T+\frac{1}{1+\mu}D_\mu$ and
$\frac{\mu}{1+\mu}+\frac{1}{1+\mu}=1$.
\end{proof}

Despite its na\"ive appearance, Lemma~\ref{lemma:convexity} is a very
handy tool. To illustrate this, we here verify Theorem~\ref{theorem:technical} for a del Pezzo surface of degree $1$.
This simple case also  immediately follows from the proof of
\cite[Lemma~3.1]{Ch07a} or from the proof of
\cite[Proposition~5.1]{KPZ12b}.

\begin{lemma}
\label{lemma:dP1} Suppose that $S$ is a smooth del Pezzo surface of degree $1$ and $D$ is an effective anticanonical $\mathbb{Q}$-divisor on $S$.  If the log pair $(S,D)$ is not log
canonical at the point $P$, then
there exists a  unique divisor $T\in |-K_{S}|$ such that $(S,T)$ is not
log canonical at $P$. Moreover, the support of $D$  contains all the
irreducible components of $T$.
\end{lemma}

\begin{proof}
Let $T$ be a curve in $|-K_{S}|$ that passes through the point $P$. Note
that $T$ is irreducible. If the log pair $(S,T)$ is log canonical at $P$, then
it follows from Lemma~\ref{lemma:convexity} that there exists an effective anticanonical $\mathbb{Q}$-divisor 
$D^{\prime}$ on the surface $S$ such that
 the log pair $(S,D^{\prime})$
is not log canonical at $P$ and $\mathrm{Supp}(D^{\prime})$ does not
contain the curve $T$. We then obtain $1=T\cdot
D^{\prime}\geqslant\mathrm{mult}_{P}(D^{\prime})$. This is
impossible by Lemma~\ref{lemma:mult}. Thus, the log pair $(S,T)$ is
not log canonical at~$P$. 

Moreover, \red{by Lemma~\ref{lemma:mult}} the divisor $T$ is singular at the point $P$. Therefore, the point $P$ is not the base point of the pencil $|-K_S|$. Consequently, such a divisor $T$ is unique.

If the curve $T$ is not
contained in $\mathrm{Supp}(D)$, then  we obtain an absurd inequality $1=T\cdot
D\geqslant \red{\mathrm{mult}_{P}(D)>1}$. Therefore, the curve $T$ must be contained in $\mathrm{Supp}(D)$ by Lemma~\ref{lemma:mult}.
\end{proof}

The following is a ready-made Adjunction for our situation. 
\red{\begin{lemma}
\label{lemma:adjunction} 
Suppose that the log pair $(S,D)$ is not
log canonical at the smooth point $P$.
If a component $D_j$ with $a_j\leqslant 1$ is smooth at $P$, then
$$
D_{j}\cdot\left(\sum_{i\ne j}a_{i}D_{i}\right)\geqslant\sum_{i\ne j }a_{i}\left(D_{j}\cdot D_{i}\right)_P>1,%
$$
where $\left(D_{j}\cdot D_{i}\right)_P$ is the local intersection number of $C_i$ and $C_j$ at $P$.
\end{lemma}
\begin{proof}
It immediately follows from \cite[Theorem 5.50]{KoMo}.
\end{proof}}

Let $f\colon \tilde{S}\to S$ be the blow up of  the surface $S$ at the point $P$ with the exceptional divisor $E$ and let
$\tilde{D}$ be the proper transform of $D$ by the blow up $f$. Then
$$
K_{\tilde{S}}+\tilde{D}+\left(\mathrm{mult}_{P}(D)-1\right)E=f^{*}\left(K_{S}+D\right).
$$
The log pair $(S,D)$ is log canonical at  $P$ if
and only if the log pair $(\tilde{S},
\tilde{D}+(\mathrm{mult}_{P}(D)-1)E)$ is log canonical along the curve $E$.

\begin{remark}
\label{remark:log-pull-back} If the log pair $(S,D)$ is not log canonical at
$P$, then there exists a point $Q$ on $E$ at which the log pair
$(\tilde{S}, \tilde{D}+(\mathrm{mult}_{P}(D)-1)E)$ is not log
canonical. Lemma~\ref{lemma:mult} then implies
\red{\begin{equation}\label{equation:greater-than-2}
\mathrm{mult}_{P}(D)+\mathrm{mult}_{Q}(\tilde{D})>2.
\end{equation}}
 If $\mathrm{mult}_{P}(D)\leqslant 2$,
then the log pair $(\tilde{S},
\tilde{D}+(\mathrm{mult}_{P}(D)-1)E)$ is log canonical at every
point of the curve $E$ other than the point $Q$. Indeed, if
 the log pair $(\tilde{S},
\tilde{D}+(\mathrm{mult}_{P}(D)-1)E)$ is not log canonical at another point $O$ on $E$, then Lemma~\ref{lemma:adjunction} generates an absurd inequality 
$$
2\geqslant\mathrm{mult}_{P}(D)=\tilde{D}\cdot E\geqslant\mathrm{mult}_{Q}(\tilde{D})+\mathrm{mult}_{O}(\tilde{D})>2.%
$$
\end{remark}

\begin{notation}From now on, when we have a birational morphism of a surface denoted by a capital roman character with tilde onto a surface, in order to denote the proper transform of a divisor by this morphism, we will add tilde to the same character that denotes the original divisor. For example, in the similar situation as the one preceding Remark~\ref{remark:log-pull-back}, we use $\tilde{D}$ for the proper transform of $D$ by $f$ without mentioning.
\end{notation}

\section{Del Pezzo surfaces of degree $2$}
\label{sec:dP2}

Let $S$ be a del Pezzo surface of degree $2$ with at most two ordinary double points. Then the linear system
$|-K_{S}|$ is base-point-free and induces a double cover
$\pi\colon S\to \mathbb{P}^2$ ramified along a reduced
quartic curve $R\subset\mathbb{P}^2$. Moreover, the curve $R$ has
at most two ordinary double points. In particular, the quartic curve $R$ is irreducible.

\begin{lemma}
\label{lemma:dP2-finitely-many} For an effective anticanonical $\mathbb{Q}$-divisor  $D$ on
$S$, the log pair $(S,D)$ is log canonical outside   finitely many points on
$S$.
\end{lemma}

\begin{proof}
\red{Suppose the converse.} Then we may write  $D=a_1C_1+\Omega$, where
 $C_1$ is an irreducible reduced
curve, $a_1$ is  a positive rational number strictly bigger than $1$
and $\Omega$ is an effective $\mathbb{Q}$-divisor whose support does
not contain the curve $C_1$. Since
$$
2=-K_{S}\cdot D=-K_{S}\cdot\left(a_1C_1+\Omega\right)=-a_1K_{S}\cdot C_1-K_{S}\cdot\Omega \geqslant -a_1K_{S}\cdot C_1>-K_{S}\cdot C_1,%
$$
we have $-K_{S}\cdot C_1=1$. Then $\pi(C_1)$ is a line in
$\mathbb{P}^2$. Thus, there exists an irreducible reduced curve
$C_2$ on $S$ such that $C_1+C_2\sim -K_{S}$ and
$\pi(C_1)=\pi(C_2)$. Note that $C_1=C_2$ if and only if
the line $\pi(C_1)$ is an irreducible component of the branch
curve $R$. Since   $R$ is
irreducible, this is not the case. Thus, we have $C_1\ne
C_2$.

Note that $C_1^{2}=C_2^2$ because $C_1$ and $C_2$ are
interchanged by the biregular involution of $S$ induced by the
double cover $\pi$. Thus, we have
$$
2=\red{(-K_S)^2}=\left(C_1+C_2\right)^2=2C_1^2+2C\cdot C_2,
$$
which implies that $C_1\cdot C_2=1-C_1^2$. Since $C_1$ and
$C_2$ are smooth rational curves, we can easily obtain $
C_1^{2}=C_2^2=-1+\frac{k}{2}$, where $k$ is the number of singular points of $S$ that lie on  $C_1$.

Now we write $D=a_1C_1+a_2C_2+\Gamma$, where $a_2$ is a
non-negative rational number and $\Gamma$ is an effective
$\mathbb{Q}$-divisor whose support contains neither  $C_1$
nor  $C_2$. Then
\[\begin{split}
1 &=C_1\cdot\left(a_1C_1+a_2C_2+\Gamma\right)
=a_1C_1^2+a_2C_1\cdot C_2+C_1\cdot\Gamma \\
& \geqslant  a_1C_1^2+a_2C_1\cdot C_2 =a_1C_1^2+a_2(1-C_1^2),\\
\end{split}
\]
and hence $1\geqslant a_1C_1^2+a_2(1-C_1^2)$. Similarly,
from $C_2\cdot D=1$, we obtain
$1\geqslant a_2C_1^2+a_1(1-C_1^2)$. 
The obtained two inequalities imply that $a_1\leqslant 1$ and $a_2\leqslant 1$ 
since $C_1^2=-1+\frac{k}{2}$, $k=0,1,2$. Since $a_1>1$ by our assumption, \red{this} is a contradiction.
\end{proof}

The following two lemmas can be verified in a similar way  as that of \cite[Lemma~3.5]{Ch07a}. Nevertheless we present their proofs
since we should carefully deal with singular points on $S$ that have been considered neither in \cite{Ch07a}  nor in \cite{KPZ12b}.

\begin{lemma}
\label{lemma:dP2-du-Val} For \red{any} effective anticanonical  $\mathbb{Q}$-divisor  $D$ on
$S$, the log pair $(S,D)$ is log
canonical at every point outside the ramification divisor of  the double cover  $\pi$.
\end{lemma}

\begin{proof}
Suppose that $(S,D)$ is not log canonical at a point $P$ whose image by $\pi$ lies outside~$R$. 

Let $H$ be a general curve in $|-K_{S}|$ that passes through the
point $P$. Since $\pi(P)\not\in R$, the surface $S$ is smooth at
the point $P$. Then
$$
2=H\cdot
D\geqslant\mathrm{mult}_{P}(H)\mathrm{mult}_{P}(D)\geqslant \mathrm{mult}_{P}(D),%
$$
and hence $\mathrm{mult}_{P}(D)\leqslant 2$.

Let $f\colon \tilde{S}\to S$ be the blow up of the surface $S$ at $P$.
We have
$$
K_{\tilde{S}}+\tilde{D}+\left(\mathrm{mult}_{P}(D)-1\right)E=f^{*}\left(K_{S}+D\right),
$$
where  $E$ is the exceptional curve of the blow up $f$.
Then, Remark~\ref{remark:log-pull-back} gives a unique point $Q$ on $E$ such that 
 the log pair $(\tilde{S},
\tilde{D}+\left(\mathrm{mult}_{P}(D)-1\right)E)$ is not log canonical
at $Q$ on  $E$.

Since $\pi(P)\not\in R$, there exists a unique reduced  but possibly
reducible curve $C\in|-K_{S}|$ such that  $C$
passes through  $P$ and its proper transform $\tilde{C}$  passes through the point $Q$. Note that  $C$ is smooth at  $P$.
 Since $(S,C)$ is log canonical at $P$,  Lemma~\ref{lemma:convexity} enables us to assume that the support of $D$ does not contain at least one irreducible component of  $C$.

If the curve $C$ is irreducible, then
$$
2-\mathrm{mult}_{P}(D)=2-\mathrm{mult}_{P}(C)\mathrm{mult}_{P}(D)=\tilde{C}\cdot\tilde{D}\geqslant\mathrm{mult}_{Q}(\tilde{C})\mathrm{mult}_{Q}(\tilde{D})=\mathrm{mult}_{Q}(\tilde{D}).
$$
This contradicts (\ref{equation:greater-than-2}).
Thus, the curve $C$ must be reducible.

We may then write
$C=C_1+C_2$, where $C_1$ and $C_2$ are irreducible smooth curves
that intersect  at two points. Without loss of generality we may assume that the curve $C_1$ is not
contained in the support of $D$.  The point $P$ must belong to
$C_2$: otherwise we would have
\[1=D\cdot C_1\geqslant \mult_{P}(D)>1.\]
We put $D=aC_2+\Omega$, where $a$ is a non-negative rational
number and $\Omega$ is an effective $\mathbb{Q}$-divisor whose
support does not contain the curve $C_2$. Then
$$1=C_1\cdot D=(2-\frac{1}{2}k)a+C_1\cdot\Omega\geqslant (2-\frac{1}{2}k)a,$$
where $k$ is the number of singular points of $S$ on $C_1$.
 On the other hand, the log pair $(\tilde{S},
a\tilde{C}_2+\tilde\Omega+(\mult_P(D)-1)E)$ is not log canonical at $Q$,
where  we have $a\leqslant 1$ by Lemma~\ref{lemma:dP2-finitely-many}. We then obtain
$$(2-\frac{1}{2}k)a=\tilde{C}_2\cdot (\tilde{\Omega}+(\mult_P(D)-1)E)>1$$
from Lemma~\ref{lemma:adjunction}.
This is a contradiction. 
\end{proof}

\begin{lemma}
\label{lemma:dP2-double-points} For  a smooth point $P$ of $S$ with $\pi(P)\in R$, let $T_P$ be the
unique divisor in $|-K_{S}|$ that is singular at the point $P$.
If the log pair $(S,T_P)$ is log canonical at $P$, then for \red{any} effective anticanonical $\mathbb{Q}$-divisor  $D$ on
$S$  the log pair $(S,D)$ is log canonical at~$P$.
\end{lemma}

\begin{proof}
Suppose that $(S,D)$ is not log canonical at the point $P$.  Applying Lemma~\ref{lemma:convexity} to the log
pairs $(S,D)$ and $(S,T_P)$, we may assume that $\mathrm{Supp}(D)$
does not contain at least one irreducible component of 
$T_P$. Thus, if the divisor $T_P$ is irreducible, then Lemma~\ref{lemma:mult} gives an absurd inequality 
$$
2=T_P\cdot
D\geqslant\mathrm{mult}_{P}(T_{P})\mathrm{mult}_{P}(D)\geqslant
2\mathrm{mult}_{P}(D)>2
$$
since $T_P$ is singular at 
$P$. Hence,  $T_P$ must be reducible.

We may then write $T_P=T_1+T_2$, where $T_1$ and $T_2$
are smooth rational curves. Note that the point $P$ is one of the intersection points of 
$T_1$ and $T_2$. Without loss of generality, we may
assume that the curve $T_1$ is not contained in the support of $D$. Then
$$
1=T_1\cdot
D\geqslant\mathrm{mult}_{P}(T_{1})\mathrm{mult}_{P}(D)=\mathrm{mult}_{P}(D)>1
$$
by Lemma~\ref{lemma:mult}. The obtained contradiction completes
the proof.
\end{proof}

Lemmas~\ref{lemma:dP2-du-Val} and
\ref{lemma:dP2-double-points} prove the following result. 

\begin{lemma}
\label{lemma:dP2} Suppose that the del Pezzo surface $S$ \red{of degree $2$} is smooth. Let $D$ be an effective anticanonical $\mathbb{Q}$-divisor on $S$. Suppose that the log pair $(S,D)$ is
not log canonical at a point $P$. Then there exists a unique divisor
$T\in |-K_{S}|$ such that $(S,T)$ is not log canonical at $P$. 
The support of the divisor
$D$ contains all the irreducible components of $T$.
The divisor $T$ is either an irreducible rational curve with a cusp at $P$ or a union of two $(-1)$-curves meeting tangentially at~$P$.
\end{lemma}

\begin{proof}
By Lemma~\ref{lemma:dP2-du-Val}, the point $\pi(P)$ must lie on $R$. Then there exists
a unique curve $T\in |-K_{S}|$ that is singular at the point $P$.
By Lemma~\ref{lemma:dP2-double-points}, the log pair $(S,T)$ is
not log canonical at $P$.

Suppose that the support of $D$ does not contain an
irreducible component of $T$. Then the proof of 
Lemma~\ref{lemma:dP2-double-points} works verbatim to derive a contradiction. 

The last assertion immediately follows from \cite[Proposition~3.2]{Pa01}.
\end{proof}
Lemma~\ref{lemma:dP2} shows that Theorem~\ref{theorem:technical} holds for a smooth del Pezzo surface of degree~$2$.

\section{Cubic surfaces}
\label{sec:cubic-surfaces}

In the present  section we prove Theorem~\ref{theorem:technical}.
 Lemma~\ref{lemma:dP1} and Lemma~\ref{lemma:dP2} show that Theorem~\ref{theorem:technical}  holds for del Pezzo surfaces of degrees $1$ and $2$, respectively.  Thus, to complete the proof,  we let $S$ be a 
smooth cubic surface in $\mathbb{P}^3$ and let $D$ be an effective anticanonical
$\mathbb{Q}$-divisor on  $S$.

\begin{lemma}
\label{lemma:cubic-finitely-many-points} The log pair $(S,D)$ is
log canonical outside finitely many points.
\end{lemma}

\begin{proof}
 Suppose not. Then we may write  $D=aC+\Omega$, where
 $C$ is an irreducible curve, $a$ is a  positive rational number   strictly bigger than $1$ and $\Omega$ is an effective $\mathbb{Q}$-divisor  whose
support does not contain the curve $C$. Then
$$
3=-K_{S}\cdot\left(aC+\Omega\right)=-aK_{S}\cdot C-K_{S}\cdot\Omega \geqslant -aK_{S}\cdot C>-K_{S}\cdot C.%
$$
This implies that  $C$ is either a line or an irreducible conic.

Suppose that $C$ is a line. Let $Z$ be a general
irreducible conic on $S$ such that $Z+C\sim -K_{S}$.  Since $Z$ is general, it is not contained in the support of $D$. 
We then obtain $$
2=Z\cdot D=Z\cdot\left(aC+\Omega\right)=2a+Z\cdot\Omega\geqslant 2a.%
$$
\red{This} contradicts our assumption.

Suppose that $C$ is an irreducible  conic. Then there exists a unique line
$L$ on $S$ such that $L+C\sim -K_{S}$. Write $D=aC+bL+\Gamma$,
where $b$ is a non-negative rational number and $\Gamma$ is an
effective $\mathbb{Q}$-divisor whose support contains neither the
conic $C$ nor the line $L$. Then
$$
1=L\cdot D=L\cdot\left(aC+bL+\Gamma\right)=2a-b+L\cdot\Gamma\geqslant 2a-b.%
$$
On the other hand,
$$
2=C\cdot D=C\cdot\left(aC+bL+\Gamma\right)=2b+C\cdot\Gamma\geqslant 2b.%
$$
Combining these two inequalities, we obtain $a\leqslant
 1$. This contradicts our
assumption too.
\end{proof}

For a point $P$ on  $S$, let $T_P$ be the tangent hyperplane
section of the surface $S$  at the point $P$. 
This is the unique anticanonical divisor that is singular at $P$.
 The curve $T_P$ is reduced but
it may be reducible.

In order to prove Theorem~\ref{theorem:technical} we must show that
$(S,D)$ is log canonical at $P$ provided that one of the following
two conditions is satisfied: 
\begin{itemize}
\item the log pair $(S,T_P)$ is log
canonical at $P$; 
\item the log pair $(S,T_P)$ is not log canonical at
$P$ but $\mathrm{Supp}(D)$ does not contain at least one
irreducible component of  $T_P$.
\end{itemize}

 The log pair $(S,T_P)$ is
log canonical at $P$ if and only if the point $P$ is an ordinary double
point of  $T_P$ \red{(see \cite[Proposition~3.2]{Pa01})}. Thus, $(S,T_P)$ is log canonical at $P$
if and only if $T_P$ is one of the following curves: an
irreducible cubic curve with one ordinary double point, a union of
three coplanar lines that do not intersect at one point, a union
of a line and a conic that intersect transversally at two points.

Overall, we must consider the following cases:
\begin{enumerate}
\item[($\mathrm{a}$)] $T_P$  is a union of three lines that intersect at  $P$ (Eckardt point);%
\item[($\mathrm{b}$)] $T_P$ is a union of a line and a conic that intersect tangentially at  $P$;
\item[($\mathrm{c}$)] $T_P$ is an irreducible cubic curve with a cusp at  $P$;
\item[($\mathrm{d}$)] $T_P$ is an irreducible cubic curve with one ordinary double point;
\item[($\mathrm{e}$)] $T_P$ is a union of three coplanar lines that do not intersect at one point;
\item[($\mathrm{f}$)] $T_P$ is a union of a line and a conic that intersect transversally at two points.%
\end{enumerate}

We consider these cases one by one in separate lemmas, i.e.,
Lemmas~\ref{lemma:cubic-Eckard},~\ref{lemma:cubic-tacnodal}, ~\ref{lemma:cubic-cusp},~\ref{lemma:cubic-ODP}, ~\ref{lemma:cubic-triangle} and
\ref{lemma:cubic-nodal}. We however present
the detailed proof of Lemma~\ref{lemma:cubic-triangle}  in
Section~\ref{sec:three-lines} to improve the readability of this
section.  These
lemmas altogether imply Theorem~\ref{theorem:technical}. 

For simplicity, put $m= \mathrm{mult}_{P}(D)$.
\begin{lemma}\label{lemma:cubic-line} 
\red{If the log pair $(S, D)$ is not log canonical at the point $P$, then the support of $D$ contains all the lines on $S$ passing through $P$.}
\end{lemma}
\begin{proof}
Let $L$ be a line passing through the point $P$ that is not contained in the support of $D$.
Then the inequality $1=L\cdot
D\geqslant m$  implies that the log pair $(S, D)$ is log canonical at  $P$  by
Lemma~\ref{lemma:mult}.
\end{proof}

\begin{lemma}[{\cite[Lemma~4.13]{KPZ11a}}]
\label{lemma:cubic-Eckard} Suppose that  the tangent hyperplane section $T_P$  consists of three
lines intersecting at the point $P$. If the support of $D$  does not contain at least one of the three lines, then the log pair $(S,D)$ is log canonical at the
point $P$.
\end{lemma}
\begin{proof}
It immediately follows from Lemma~\ref{lemma:cubic-line}.\end{proof}

From now on, let $f:\tilde{S}\to S$ be the blow up of the cubic surface $S$ at the point $P$.  In addition, let   $E$ be the exceptional curve of $f$.  We then have 
\begin{equation*}
K_{\tilde{S}}+\tilde{D}+\left(m-1\right)E=f^{*}\left(K_{S}+D\right).%
\end{equation*}
Note that the log pair $(S, D)$ is log canonical at  $P$ if and only if the log pair $$(\tilde{S}, \tilde{D}+\left(m-1\right)E)$$ is log canonical along the exceptional divisor $E$.

\begin{remark}\label{remark:singular-del-Pezzo}
If there is a line passing through $P$, then the surface $\tilde{S}$ is a weak del Pezzo surface of degree $2$, i.e.,  $K_{\tilde{S}}^2=2$ and $-K_{\tilde{S}}$ is nef and big. The proper transforms of the lines passing though $P$ will be $(-2)$-curves on $\tilde{S}$. All the $(-2)$-curves on  $\tilde{S}$ are disjoint each other and they come from the lines passing through $P$ on $S$. By contracting these $(-2)$-curves we obtain a birational morphism  $g:\tilde{S}\to \bar{S}$.
Then $\bar{S}$ is a del Pezzo surface of degree $2$ with 
ordinary double points. In particular,
the linear system $|-K_{\bar{S}}|$ 
induces a double cover $\pi\colon \bar{S}\to \mathbb{P}^2$ 
ramified along a  quartic curve
$R\subset\mathbb{P}^2$. 
The $(-2)$-curves on $\tilde{S}$ are contracted to the ordinary double points on $\bar{S}$. Therefore,  the number of the ordinary double points on $\bar{S}$ is given by the number of lines passing through $P$ on $S$.  Since we have at most two lines passing though $P$, the surface $\bar{S}$ has at most two ordinary double points, and hence the quartic curve $R$ must be an  irreducible curve with at most two ordinary double points.
\end{remark}

\begin{lemma}
\label{lemma:cubic-tacnodal} Suppose that the tangent hyperplane section $T_P$  consists of a
line and a conic intersecting tangentially at the point $P$.
If the support of $D$  does not contain \red{both  the line and the conic}, then the log pair $(S,D)$ is log canonical at the
point $P$.
\end{lemma}

\begin{proof}
Suppose that the log pair $(S,D)$ is not log canonical at the point $P$. Let $L$ and $C$ be the line and the conic, respectively, such that $T_P=L+C$. 
By Lemma~\ref{lemma:cubic-line}, we may assume that the conic $C$ is not contained  but the line $L$  is contained  in the support of $D$.
We write $D=aL+\Omega$, where $a$ is a
positive rational number and $\Omega$ is an effective
$\mathbb{Q}$-divisor whose support contains neither  the
line $L$ nor the conic $C$. 
We have 
$m\leqslant C\cdot D=2$. 

Note that the three curves $\tilde{L}$, $\tilde{C}$ and $E$ meet \red{at one point transversally}.
Since $m\leqslant 2$, we have the unique point $Q$ on $E$ defined in Remark~\ref{remark:log-pull-back}.
The point $Q$  does not belong to $\tilde{C}$, and hence not to $\tilde{L}$ either.
 Indeed, \red{otherwise}
$$
2-m=\tilde{C}\cdot\left(a\tilde{L}+\tilde{\Omega}\right)\geqslant a+\mathrm{mult}_{Q}(\tilde{\Omega})=\mult_Q(\tilde{D}).%
$$
This contradicts (\ref{equation:greater-than-2}).

Let $g\colon\tilde{S}\to\bar{S}$ be the contraction  defined in Remark~\ref{remark:singular-del-Pezzo}.  Note that the point $g(\tilde{L})$
is the ordinary double point of the surface $\bar{S}$. Put
$\bar{\Omega}=g(\tilde{\Omega})$, $\bar{E}=g(E)$, 
$\bar{C}=g(\tilde{C})$ and  $\bar{Q}=g(Q)$. Then
$\pi(\bar{E})=\pi(\bar{C})$ since $\bar{E}+\bar{C}$ is an anticanonical divisor on $\bar{S}$.
The point 
$\pi(\bar{Q})$ lies outside $R$ because the point $Q$ lies outside $\tilde{C}$.
Since the divisor 
$\bar{\Omega}+\big(m-1\big)\bar{E}$ is $\mathbb{Q}$-linearly equivalent to $-K_{\bar{S}}$
by our construction, Lemma~\ref{lemma:dP2-du-Val} shows that 
 the log pair $(\bar{S},
\bar{\Omega}+(m-1)\bar{E})$ is  log canonical at $\bar{Q}$.
However, it is not log
canonical at the point $\bar{Q}$ since $g$ is an isomorphism in
a neighborhood of the point $Q$.   \red{This} is a contradiction.
\end{proof}

\begin{lemma}
\label{lemma:cubic-cusp} Suppose that the tangent hyperplane section $T_P$ is an irreducible cubic curve with a cusp at $P$. 
If  $T_P$ is not contained in the support of $D$, then the log pair $(S,D)$ is log canonical at
 $P$.
\end{lemma}

\begin{proof}
Suppose that $(S,D)$ is not log canonical at $P$. 
From the inequality 
$$
3=T_P\cdot
D\geqslant m\cdot\mathrm{mult}_{P}(T_P)=2m,
$$
we obtain  $m\leqslant \frac{3}{2}$.
Then, we have the unique point $Q$ on $E$ defined in Remark~\ref{remark:log-pull-back}.

The surface $\tilde{S}$ is a smooth del Pezzo surface of degree $2$. The linear system
$|-K_{\tilde{S}}|$ induces a double
cover $\pi\colon \tilde{S}\to \mathbb{P}^2$ ramified along a
smooth quartic curve $R\subset\mathbb{P}^2$.  Then the integral  divisor $E+\tilde{T}_P$ is linearly equivalent to $-K_{\tilde{S}}$, and hence 
$\pi(E)=\pi(\tilde{T}_P)$ is a line in $\mathbb{P}^2$.  
Moreover,
$\tilde{T}_{P}$ tangentially meet $E$   at  a single point.
Thus the point $\pi(Q)$ lies on $R$ if and only if
the point $Q$ is the intersection point of $E$ and $\tilde{T}_P$.

Applying Lemma~\ref{lemma:dP2-du-Val} to the
log pair $(\tilde{S},
\tilde{D}+\big(m-1\big)E)$, we see that the point 
$\pi(Q)$ belongs to $R$ because the log pair $(\tilde{S},
\tilde{D}+(m-1)E)$ is not log canonical at the point $Q$
and the divisor
$\tilde{D}+(m-1)E$ is $\mathbb{Q}$-linearly  equivalent to $-K_{\tilde{S}}$. The point $Q$ therefore lies on the curve $\tilde{T}_P$.
 Then from (\ref{equation:greater-than-2}) we obtain 
$$
3-2m=\tilde{T}_{P}\cdot\tilde{D}\geqslant\mathrm{mult}_{Q}(\tilde{D})>2-m.
$$
This contradicts Lemma~\ref{lemma:mult}.
\end{proof}

For the remaining three cases, we show that the hypothesis of Theorem~\ref{theorem:technical} is never fulfilled, so that Theorem~\ref{theorem:technical} is true.
\begin{lemma}
\label{lemma:cubic-ODP} If the tangent hyperplane section $T_P$ is an irreducible cubic curve with a node at~$P$, then  the log pair $(S,D)$ is log canonical at
$P$.
\end{lemma}

\begin{proof}
Suppose that $(S,D)$ is not log canonical at $P$. 
The surface $\tilde{S}$ is a smooth del Pezzo surface of degree two.
Since $\tilde{D}+(m-1)E\qsim -K_{\tilde{S}}$ and the log pair $(\tilde{S},
\tilde{D}+(m-1)E)$  is not log canonical at some point $Q$ on $E$,
it follows from Lemma~\ref{lemma:dP2} that there
must be an anticanonical divisor $H$ on the surface $\tilde{S}$ that
has either a tacnode or a cusp at the point~$Q$.

If the divisor $H$ has a tacnode at  $Q$, then it
consists of the exceptional divisor $E$ and another $(-1)$-curve $L$
meeting $E$ tangentially at $Q$. Then the divisor $f(H)$ is an
effective anticanonical divisor on $S$ such that it has a cusp at
$P$  and it is distinct from the divisor $T_P$. This is
impossible.

If the divisor $H$ has a cusp at the point $Q$, then it must be
irreducible. However, it is impossible  since $H$ is singular at
$Q$ and $E\cdot H=1$.
\end{proof}

\begin{lemma}
\label{lemma:cubic-triangle} Suppose that the tangent hyperplane section $T_{P}$ consists of
three lines  one of which does not pass through the point~$P$. Then the log pair $(S,D)$ is log canonical at~$P$.
\end{lemma}

\begin{proof}
The proof of this lemma is  the central and the most beautiful
part of the proof of Theorem~\ref{theorem:technical}. \red{Since it is a bit lengthy, it will be presented in a separate section.
See Section~\ref{sec:three-lines}.}
\end{proof}

\begin{lemma}
\label{lemma:cubic-nodal} Suppose that the tangent hyperplane section $T_P$ consists of a line
and a conic intersecting transversally. Then the log pair 
$(S,D)$ is log canonical at the point $P$.
\end{lemma}

\begin{proof}
We write $T_P=L+C$, where $L$ is a line and $C$ is an irreducible
conic that intersect $L$ transversally at  $P$.  Suppose that $(S,D)$ is not log canonical at  $P$.

By Lemmas~\ref{lemma:convexity} and \ref{lemma:cubic-line}, we may assume that the conic $C$ is not contained  but the line $L$  is contained  in the support of $D$.
We write $D=aL+\Omega$, where $a$ is a
positive rational number and $\Omega$ is an effective
$\mathbb{Q}$-divisor whose support contains neither  the
line $L$ nor the conic $C$. 

We have the unique point $Q$ on $E$ defined in Remark~\ref{remark:log-pull-back}  since $m\leqslant D\cdot C=2$.


Suppose that the point $Q$ does not belong to the $(-2)$-curve $\tilde{L}$.
Let $g\colon\tilde{S}\to\bar{S}$ be the contraction  defined in Remark~\ref{remark:singular-del-Pezzo}. Then $\bar{S}$ is a del Pezzo surface of degree $2$ with only one 
ordinary double point at the point $g(\tilde{L})$. Put
$\bar{\Omega}=g(\tilde{\Omega})$, $\bar{E}=g(E)$, 
$\bar{C}=g(\tilde{C})$ and  $\bar{Q}=g(Q)$. Then
$\pi(\bar{E})=\pi(\bar{C})$ since $\bar{E}+\bar{C}$ is an anticanonical divisor on $\bar{S}$.
The point 
$\pi(\bar{Q})$ lies on $R$ if and only if the point $Q$ lies on $\tilde{C}$.
The log pair $(\bar{S},
\bar{\Omega}+(m-1)\bar{E})$ is  not log canonical at $\bar{Q}$
since $g$ is an isomorphism in
a neighborhood of the point $Q$. 
Since the divisor 
$\bar{\Omega}+\big(m-1\big)\bar{E}$ is $\mathbb{Q}$-linearly equivalent to $-K_{\bar{S}}$
by our construction, Lemma~\ref{lemma:dP2-du-Val} shows that the point $Q$ belongs to $\tilde{C}$.

Note that $\bar{C}+\bar{E}$ is
the unique curve in $|-K_{\bar{S}}|$ that is singular at 
$\bar{Q}$. But the log pair $(\bar{S},\bar{C}+\bar{E})$ is log
canonical at  $\bar{Q}$. Hence, it follows from
Lemma~\ref{lemma:dP2-double-points} that  $(\bar{S},
\bar{\Omega}+(m-1)\bar{E})$ is log canonical at
$\bar{Q}$. This is a contradiction. Therefore, the point $Q$ must belong to the $(-2)$-curve $\tilde{L}$.

Now we can apply  \cite[Theorem~1.28]{ChK10} to the log pair
$(\tilde{S}, a\tilde{L}+(m-1)E+\tilde{\Omega})$
at the point $Q$ to obtain a contradiction  immediately. Indeed, it
is enough to put $M=1$, $A=1$, $N=0$, $B=2$, and $\alpha=\beta=1$
in \cite[Theorem~1.28]{ChK10} and check that all the conditions of
\cite[Theorem~1.28]{ChK10} are satisfied. However, there is a much
simpler way to obtain a contradiction. Let us take this simpler way.

There exists another line $M$ on the surface $S$ that intersects $L$ at a point. The line $M$ does not intersect the conic $C$ since  $1=T_P\cdot M=(L+C)\cdot
M=L\cdot M$. In particular, the point $P$ does not lie on the line $M$.
   Let
$h\colon\tilde{S}\to\check{S}$ be the contraction of the proper transform of the line $M$ on the surface $\tilde{S}$.
Since $M$ is a $(-1)$-curve and the point $P$ does not lie on  $M$, the surface $\check{S}$ is a smooth cubic surface in
$\mathbb{P}^3$.

Put $\check{\Omega}=h(\tilde{\Omega})$, $\check{E}=h(E)$, 
$\check{L}=h(\tilde{L})$, $\check{C}=h(\tilde{C})$, 
$\check{P}=h(Q)$  and $\check{D}=h(\tilde{D})$. Then
$(\check{S},\check{D})$ is not log canonical at the point
$\check{P}$ since $h$ is an isomorphism in a neighborhood of
the point $Q$. On the other hand, the divisor $\check{L}+\check{C}+\check{E}$ is an anticanonical divisor of the surface $\check{S}$.
Since the point $\check{P}$ is the intersection point of $\check{L}$ and $\check{E}$ and the divisor $\check{D}$ is $\mathbb{Q}$-linearly equivalent to $-K_{\check{S}}$, 
Lemma~\ref{lemma:cubic-triangle} implies that 
$(\check{S},\check{D})$ is log canonical at 
$\check{P}$. This is a contradiction.
\end{proof}

As we already mentioned, Theorem~\ref{theorem:technical} follows from
Lemmas~\ref{lemma:cubic-Eckard},~\ref{lemma:cubic-tacnodal}, ~\ref{lemma:cubic-cusp},~\ref{lemma:cubic-ODP}, ~\ref{lemma:cubic-triangle} and
\ref{lemma:cubic-nodal}. Thus
Theorem~\ref{theorem:technical} has been proved  under the assumption that  
Lemma~\ref{lemma:cubic-triangle} is valid. 
The assumption will be confirmed in the following section.

\section{The proof of Lemma~\ref{lemma:cubic-triangle}}
\label{sec:three-lines}

To prove  Lemma~\ref{lemma:cubic-triangle}, we keep the notations used in Section~\ref{sec:cubic-surfaces}.
We write 
$T_P=L+M+N$, where $L$, $M$, and $N$ are three
coplanar lines on $S$. We may assume that  the point $P$ is the intersection point of the two lines $L$ and $M$, whereas  it does not lie on the line
$N$.  We also write $D=a_0L+b_0M+c_0N+\Omega_0$, where $a_0$, $b_0$, and $c_0$ are
non-negative rational numbers and $\Omega_0$ is an effective
$\mathbb{Q}$-divisor on $S$ whose support  contains none of the lines
$L$, $M$ and $N$.  Put $m_0=\mathrm{mult}_P(\Omega_0)$.

Suppose that the log pair $(S,D)$ is not log canonical at  the point $P$. Let us look for a contradiction.

 By Lemma~\ref{lemma:cubic-finitely-many-points},
the log pair $(S,D)$ is log canonical outside  finitely many points. 
In particular, we have $0\leqslant a_0, b_0, c_0\leqslant 1$.
Also, Lemma~\ref{lemma:mult} implies  that $m_0+a_0+b_0>1$ and Lemma~\ref{lemma:cubic-line} implies that $a_0$, $b_0>0$.

\begin{lemma}
\label{lemma:three-lines-main-inequality} The inequality 
$m_0+a_0+b_0>c_0+1$ holds.
\end{lemma}

\begin{proof}
Since the log pair $(S, a_0L+b_0M+\Omega_0)$ is not log canonical \red{at the point $P$}, it follows from Lemma~\ref{lemma:adjunction}
that
$$
1+a_0-c_0=L\cdot\left(D-a_0L-c_0N\right)=L\cdot\left(b_0M+\Omega_0\right)>1,
$$
which implies $a_0>c_0$. Similarly, $b_0>c_0$. 

The log pair $(S,L+M+N)$ is log canonical.
Since the log pair $(S, a_0L+b_0M+c_0N+\Omega_0)$ is not log
canonical at $P$, it follows from Lemma~\ref{lemma:convexity}  and its proof that
the log pair
$$
\left(S,\frac{1}{1-c_0}D-\frac{c_0}{1-c_0}T_P\right)$$
is not log canonical at  $P$. 
Then Lemma~\ref{lemma:mult} shows \red{that}
\[\begin{split}
\mult_P\left(\frac{1}{1-c_0}D-\frac{c_0}{1-c_0}T_P\right)&=\mult_P\left(\frac{a_0-c_0}{1-c_0}L+\frac{b_0-c_0}{1-c_0}M+\frac{1}{1-c_0}\Omega_0\right)\\ &=\frac{a_0-c_0}{1-c_0}+\frac{b_0-c_0}{1-c_0}+\frac{m_0}{1-c_0}>1.\\%
\end{split}\]
\red{This} verifies $m_0+a_0+b_0>c_0+1$.
\end{proof}

Since the rational numbers $a_0$, $b_0$, $c_0$ are \red{smaller or equal to} $1$ and the log pair $(S, L+M+N)$ is log canonical, the effective $\mathbb{Q}$-divisor $\Omega_0$ cannot be the zero-divisor.
Let $r$ be the number of the irreducible components of the support of
the $\mathbb{Q}$-divisor $\Omega_0$. Then we write
$$\Omega_0=\sum_{i=1}^{r}e_{i}C_{i0},$$ where $e_{i}$'s are positive rational numbers and $C_{i0}$'s are
irreducible reduced curves of degrees $d_{i0}$ on  $S$.  We then see 
\begin{equation}\label{equation:three-lines-degree-D-is-3}
3=-K_{S}\cdot\left(a_0L+b_0M+c_0N+\sum_{i=1}^{r}e_{i}C_{i0}\right)=a_0+b_0+c_0+\sum_{i=1}^{r}e_{i}d_{i0}.%
\end{equation}

 We have 
 $$
K_{\tilde{S}}+a_0\tilde{L}+b_0\tilde{M}+c_0\tilde{N}+\left(a_0+b_0+m_0-1\right)E+\sum_{i=1}^{r}e_{i}\tilde{C}_{i0}=f^{*}\left(K_{S}+D\right).%
$$
Recall that $a_0+b_0+m_0=m$.

\begin{lemma}
\label{lemma:three-lines-mult-D-is-at-most-2} The inequality 
$m=a_0+b_0+m_0\leqslant 2$ holds.
\end{lemma}

\begin{proof}
It immediately follows from the three inequalities
\[\aligned
&1=L\cdot\left(a_0L+b_0M+c_0N+\Omega_0\right)=-a_0+b_0+c_0+L\cdot\Omega_0\geqslant -a_0+b_0+c_0+m_0,\\
&1=M\cdot\left(a_0L+b_0M+c_0N+\Omega_0\right)=a_0-b_0+c_0+M\cdot\Omega_0\geqslant a_0-b_0+c_0+m_0,\\
&1=N\cdot\left(a_0L+b_0M+c_0N+\Omega_0\right)=a_0+b_0-c_0+N\cdot\Omega_0\geqslant a_0+b_0-c_0.\\
\endaligned\]
\end{proof}

The
log pair
\begin{equation}
\label{equation:three-lines-log-pull-back}
\left(\tilde{S}, a_0\tilde{L}+b_0\tilde{M}+c_0\tilde{N}+\left(a_0+b_0+m_0-1\right)E+\sum_{i=1}^{r}e_{i}\tilde{C}_{i0}\right)%
\end{equation}
is not log canonical at some point $Q$ on $E$.
Since
$\mathrm{mult}_{P}(D)=a_0+b_0+m_0\leqslant 2$, it follows from
Remark~\ref{remark:log-pull-back} that   $Q$ is the only point on $E$ where the log pair
$(\ref{equation:three-lines-log-pull-back})$ fails to be log canonical.

Let $g\colon\tilde{S}\to\bar{S}$ be the contraction  defined in Remark~\ref{remark:singular-del-Pezzo}.
Then $\bar{S}$ is a del Pezzo surface of degree $2$ with two ordinary double points at the points $g(\tilde{L})$ and $g(\tilde{M})$.

\begin{lemma}
\label{lemma:three-lines-Q-is-in-L} The point $Q$ on the exceptional curve $E$ belongs to either  $\tilde{L}$ or 
$\tilde{M}$.
\end{lemma}

\begin{proof}

Suppose that the point $Q$ lies on neither  $\tilde{L}$ nor $\tilde{M}$.  Put $\bar{E}=g(E)$,
$\bar{N}=g(\tilde{N})$ and $\bar{Q}=g(Q)$. In addition, we put $\bar{C}_{i0}=g(\tilde{C}_{i0})$ for each $i$. Then
$\pi(\bar{E})=\pi(\bar{N})$. The point
$\pi(\bar{Q})$ lies outside the quartic curve $R$ since  $\bar{Q}$ is a smooth point of the anticanonical divisor $\bar{E}+\bar{N}$ on $\bar{S}$.

Since $g$ is an isomorphism in a neighborhood of the point $Q$,
the log pair
\begin{equation}
\label{equation:three-lines-log-transform}
\left(\bar{S}, c_0\bar{N}+\left(a_0+b_0+m_0-1\right)\bar{E}+\sum_{i=1}^{r}e_{i}\bar{C}_{i0}\right)%
\end{equation}
is not log canonical at the point $\bar{Q}$. 
The divisor 
$c_0\bar{N}+(a_0+b_0+m_0-1)\bar{E}+\sum_{i=1}^{r}e_{i}\bar{C}_{i0}$ is an effective anticanonical $\mathbb{Q}$-divisor
on the surface $\bar{S}$.
Hence, we are able to apply Lemma~\ref{lemma:dP2-du-Val} to the log pair
$(\ref{equation:three-lines-log-transform})$ to obtain a
contradiction.
\end{proof}

\emph{From now on we may assume that the point $Q$ is the intersection point of  $\tilde{L}$ and  $E$ without loss of generality.}
\medskip

Let $\rho\colon S\dasharrow\mathbb{P}^2$ be the linear projection
from the point $P$. Then $\rho$ is a generically $2$-to-$1$
rational map. Thus the map $\rho$ induces a birational involution
$\tau_P$ of the cubic surface $S$. The involution $\tau_P$
 is  classically known as the Geiser involution associated to the point~$P$
(see \cite{Ma67}).

\begin{remark}
\label{remark:three-lines-Geiser} By construction, the involution
$\tau_P$ is biregular outside  the union $L\cup M\cup N$.
In fact, one can show that $\tau_P$ is biregular outside  the
point $P$ and the line $N$. Moreover, one can show that
$\tau_P(L)=L$ and $\tau_P(M)=M$.
\end{remark}

For each $i$, put $C_{i1}=\tau_P(C_{i0})$ and
denote by $d_{i1}$ the degree of the curve $C_{i1}$. We then employ new effective $\mathbb{Q}$-divisors
\[\aligned
&\Omega_1=\sum_{i=1}^re_iC_{i1};\\ 
&D_1=a_1L+b_1M+c_1N+\Omega_1,\\
\endaligned
\]
where $a_1=a_0$, $b_1=b_0$ and $c_1=a_0+b_0+m_0-1$.
Note that $a_0+b_0+m_0-1>0$ by Lemma~\ref{lemma:mult} (cf.
Lemma~\ref{lemma:three-lines-main-inequality}).

\begin{lemma}
\label{lemma:three-lines-main-lemma} The divisor $D_1$ is an effective anticanonical $\mathbb{Q}$-divisor on the surface $S$. The log pair $(S,D_1)$ is not
log canonical at  the intersection point of $L$ and $N$. \end{lemma}
\begin{proof}
Let
$h\colon\tilde{S}\to S'$ be the contraction of the $(-1)$-curve
$\tilde{N}$. Then $S'$ is a smooth cubic surface in
$\mathbb{P}^3$. Put $E'=h(E)$, 
$L'=h(\tilde{L})$, $M'=h(\tilde{M})$, 
$Q'=h(Q)$ and ${C}'_{i0}=h(\tilde{C}_{i0})$ for each $i$.
Then the integral divisor $L'+M'+E'$ is an anticanonical divisor of  $S'$. In particular, the curves $L'$, $M'$ and $E'$ are coplanar lines on $S'$. Moreover, the point $Q'$ is the intersection point of $L'$ and $E'$ by the assumption right after Lemma~\ref{lemma:three-lines-Q-is-in-L}. It does not lie on the line $M'$.

Let $\iota_{P}$ be the biregular involution of the surface
$\bar{S}$  induced by the double cover $\pi$. Then
$\iota_{P}$ induces a biregular involution $\upsilon_{P}$ of the
surface $\tilde{S}$ since the surface $\tilde{S}$ is the minimal resolution
of singularities of the surface $\bar{S}$. Thus, we have a
commutative diagram
$$
\xymatrix{
\tilde{S}\ar@{->}[d]_{f}\ar@{->}[drr]_{g}\ar@{->}[rrrrrr]^{\upsilon_P}&&&&&&\tilde{S}\ar@{->}[d]^{f}\ar@{->}[dll]^{g}\\%
S\ar@/_1pc/@{-->}[drrr]_{\rho}&&\bar{S}\ar@{->}[dr]_{\pi}\ar@{->}[rr]^{\iota_P}&&\bar{S}\ar@{->}[dl]^{\pi}&&S\ar@/^1pc/@{-->}[dlll]^{\rho}.\\
&&&\mathbb{P}^2&&&}
$$ %
This shows $\tau_P=f\circ\upsilon_P\circ f^{-1}$. On
the other hand, we have $\upsilon_P(E)=\tilde{N}$ since
$\pi\circ g(E)=\pi\circ g(\tilde{N})$. This means that there
exists an isomorphism $\sigma\colon S\to S'$ that makes the
diagram
$$
\xymatrix{
\tilde{S}\ar@{->}[d]_{h}\ar@{->}[rrr]^{\upsilon_P}&&&\tilde{S}\ar@{->}[d]^{f}\\%
S'\ar@{<-}[rrr]^{\sigma}&&&S\\}
$$ %
commute. By construction, $\sigma(L)=L'$,
$\sigma(M)=M'$, $\sigma(N)=E'$, and
$\sigma(C_{i1})=C'_{i0}$ for every $i$. Recall that $Q'$ is the intersection point of $L'$ and $E'$.

 Since $h$ is an
isomorphism locally around $Q$, the log pair
$$
\left(S',
a_0L'+b_0M'+\left(a_0+b_0+m_0-1\right)E'+\sum_{i=1}^{r}e_iC'_{i0}\right)
$$
is not log canonical at  $Q'$. Since
$a_0\tilde{L}+b_0\tilde{M}+c_0\tilde{N}+\left(a_0+b_0+m_0-1\right)E+\sum_{i=1}^{r}e_i\tilde{C}_{i0}\sim_{\mathbb{Q}}-K_{\tilde{S}}$,
we have
$a_0L'+b_0M'+\left(a_0+b_0+m_0-1\right)E'+\sum_{i=1}^{r}e_iC'_{i0}\sim_{\mathbb{Q}}-K_{S'}$.
Therefore, it follows
 that
$$
a_0L+b_0M+(a_0+b_0+m_0-1)N+\sum_{i=1}^re_iC_{i1}\sim_{\mathbb{Q}}-K_{S},
$$
and the log pair
$(S,a_0L+b_0M+(a_0+b_0+m_0-1)N+\sum_{i=1}^re_iC_{i1})$ is not
log canonical at the intersection the point of $L$ and $N$. \end{proof}

Now we are able to replace the
original effective $\mathbb{Q}$-divisor $D$ by the new effective
$\mathbb{Q}$-divisor $D_1$. By
Lemma~\ref{lemma:three-lines-main-lemma}, both the
$\mathbb{Q}$-divisors have the same properties that we have been
using so far. However, the new $\mathbb{Q}$-divisor $\Omega_1$ is
\emph{slightly better} than the original one $\Omega_0$ in the sense of the following lemma.

\begin{lemma}
\label{lemma:three-lines-Geiser-lowers-degree} The degree of the $\mathbb{Q}$-divisor $\Omega_1$ is strictly smaller than the degree of  $\Omega_0$, i.e., 
$$\sum_{i=1}^{r}e_id_{i1}<\sum_{i=1}^{r}e_id_{i0}.$$
\end{lemma}

\begin{proof}
Since $D_1\qsim -K_{S}$ by
Lemma~\ref{lemma:three-lines-main-lemma}, we obtain
\[\begin{split}
3&=-K_{S}\cdot\left(a_0L+b_0M+(a_0+b_0+m_0-1)N+\sum_{i=1}^re_iC_{i1}\right)\\
&=2a_0+2b_0+m_0-1+\sum_{i=1}^{r}e_id_{i1}.\\
\end{split}
\]
On the other hand, we
have $a_0+b_0+c_0+\sum_{i=1}^{r}e_id_{i0}=3$ by
(\ref{equation:three-lines-degree-D-is-3}). Thus, we obtain
$$
\sum_{i=1}^{r}e_id_{i1}=\sum_{i=1}^{r}e_id_{i0}-\left(a_0+b_0+m_0-1-c_0\right)<\sum_{i=1}^{r}e_id_{i0}
$$
because $a_0+b_0+m_0-1-c_0>0$ by
Lemma~\ref{lemma:three-lines-main-inequality}.
\end{proof}

Repeating this process, we can  obtain a sequence of the effective anticanonical 
$\mathbb{Q}$-divisors 
$$D_k=a_kL+b_kM+c_kN+\Omega_k$$ on the surface
$S$ such that  each log pair
$(S,D_k)$ is not log canonical at one of the three intersection points
$L\cap M$, $L\cap N$ and $M\cap N$. Note that  
$$\Omega_k=\sum_{i=1}^re_iC_{ik},$$ 
where $C_{ik}$'s are irreducible reduced curves of degrees $d_{ik}$.
We then obtain a strictly decreasing  sequence of rational numbers
$$
\sum_{i=1}^{r}e_id_{i0}>\sum_{i=1}^{r}e_id_{i1}>\cdots>\sum_{i=1}^{r}e_id_{ik}>\cdots
$$
by Lemma~\ref{lemma:three-lines-Geiser-lowers-degree}. This is a contradiction 
since the subset
$$
\left\{\sum_{i=1}^{r}e_in_i\ \Big\vert\ n_1,n_2,\ldots,n_r\in\mathbb{N}\right\}\subset\mathbb{Q}%
$$
is discrete and bounded \red{from} below. \red{This} completes the proof of Lemma~~\ref{lemma:cubic-triangle}.

\section{$\alpha$-functions on smooth del Pezzo surfaces}
\label{sec:dP-alpha}

In this section, we prove Theorem~\ref{theorem:Park-Won}. Let $S_d$
be a smooth del Pezzo surface of degree~$d$. 

Before we proceed, we here make a simple but useful observation.

\begin{lemma}
\label{lemma:blow-down} Let $f:S_d\to S$ be the blow down of a $(-1)$-curve  $E$ on the del Pezzo surface $S_d$.   Then $S$ is a smooth
del Pezzo surface    and
$\alpha_{S_d}(P)\geqslant\alpha_S(f(P))$ for a point $P$ of $S_d$ outside the curve $E$.
\end{lemma}

\begin{proof}
\red{It is easy to check that $-K_S$ is ample. The second statement immediately follows from the definition of the $\alpha$-function.}
\end{proof}

We \red{have already shown} that the
$\alpha$-function $\alpha_{\mathbb{P}^2}$ of the projective plane is the constant
function with the value $\frac{1}{3}$ (see
Example~\ref{example:alpha-P^n}) and the $\alpha$-function
$\alpha_{\mathbb{P}^1\times \mathbb{P}^1}$ of the quadric surface is the constant
function with the value $\frac{1}{2}$ (see
Example~\ref{example:alpha-product}).

\begin{lemma}\label{lemma:F1}
The $\alpha$-function $\alpha_{\mathbb{F}_1}$ on the
blow-up $\mathbb{F}_1$ of $\mathbb{P}^2$ at one point is the
constant function with the value $\frac{1}{3}$.\end{lemma}

\begin{proof} Let $P$ be a given point on $\mathbb{F}_1$. Let
$\pi\colon \mathbb{F}_1\to \mathbb{P}^1$ be the $\mathbb{P}^1$-bundle
morphism onto $\mathbb{P}^1$. Let $C$ be its section with
$C^2=-1$ and let $L_P$ be the fiber of the morphism $\pi$ over
the point $\pi(P)$. Since $2C+3L_P\sim -K_{\mathbb{F}_1}$, we
have $\alpha_{\mathbb{F}_1}(P)\leqslant \frac{1}{3}$. But $\alpha(\mathbb{F}_1)=\frac{1}{3}$ by
Theorem~\ref{theorem:GAFA}. Thus, $\alpha_{\mathbb{F}_1}$ is the constant
function with the value $\frac{1}{3}$ by
Lemma~\ref{lemma:inf-sup}.
\end{proof}

The surface
$S_7$ is the blow-up of $\mathbb{P}^2$ at two distinct points  $Q_1$
and $Q_2$. Let $E$ be the proper transform of the line passing
through  $Q_1$ and $Q_2$ by the two-point blow up
$f:S_7\to\mathbb{P}^2$ with the exceptional curves $E_1$ and
$E_2$.
\begin{lemma}
\label{lemma:dP7} The $\alpha$-function  on the del
Pezzo surface  $S_7$ of degree $7$ has the following values
$$
\alpha_{S_7}(P)=\left\{%
\aligned
&1/2\ \ \ \text{if $P\not\in E$}\\
&1/3 \ \ \ \text{if $P\in E$.}\\
\endaligned\right.%
$$
\end{lemma}

\begin{proof}
Let $P$ be a point on $S$. Then $\alpha_{S_7}(P)\geqslant
\alpha(S)=\frac{1}{3}$ by Theorem~\ref{theorem:GAFA} and
Lemma~\ref{lemma:inf-sup}. 

If the point $P$ belongs to $E$, then $\alpha_{S_7}(P)\leqslant\frac{1}{3}$ 
since
$2E_1+2E_2+3E\sim -K_{S}$. Therefore, $\alpha_{S_7}(P)=\frac{1}{3}$.

Suppose that the point $P$ lies outside  $E$. Let $L$ be a line on $\mathbb{P}^2$ whose proper transform by the blow up $f$ passes through $P$. Since
$f^*(2L)+E$ is an effective anticanonical divisor passing through  $P$, we have  $\alpha_{S_7}(P)\leqslant
\frac{1}{2}$.

Let $g\colon S\to \mathbb{P}^1\times\mathbb{P}^1$ be the
birational morphism obtained by contracting the $(-1)$-curve $E$.
Then this morphism is an isomorphism around  $P$. Then
$\alpha_{S_7}(P)\geqslant\alpha_{\mathbb{P}^1\times\mathbb{P}^1}(g(P))$
by Lemma~\ref{lemma:blow-down}. Since $ \alpha_{\mathbb{P}^1\times\mathbb{P}^1}$ is the constant
function with the value $\frac{1}{2}$, we obtain $\alpha_{S_7}(P)=\frac{1}{2}$.
\end{proof}

\begin{lemma}
\label{lemma:dP6} The $\alpha$-function $\alpha_{S_6}$ on the del
Pezzo surface $S_6$ of degree $6$ is the constant function with
the value $\frac{1}{2}$.
\end{lemma}

\begin{proof}
 Let $P$ be a given point on the del Pezzo surface $S_6$. One
can easily check \red{that} $\alpha_{S_6}(P)\leqslant\frac{1}{2}$. One the other hand,
we have a birational morphism $h:S_6\to S_7$, where $S_7$ is a
del Pezzo surface of degree $7$, such that the morphism $h$ is
an isomorphism around the point $P$ and the point $h(P)$ is not
on the $(-1)$-curve of $S_7$ connected to two different $(-1)$-curves.
Then
$\alpha_{S_6}(P)\geqslant\frac{1}{2}$ by
Lemmas~\ref{lemma:blow-down} and \ref{lemma:dP7}.
\end{proof}

\begin{lemma}
\label{lemma:dP5}The $\alpha$-function  on a del
Pezzo surface  $S_5$ of degree $5$ has the following values
$$
\alpha_{S_5}(P)=\left\{%
\aligned
&1/2 \ \ \ \text{if there is a $(-1)$-curve passing through  $P$;}\\
&2/3\ \ \ \text{if there is no $(-1)$-curve passing though  $P$.}\\
\endaligned\right.%
$$
\end{lemma}

\begin{proof}
Let $P$ be a point on $S_5$. Suppose that $P$ lies on a $(-1)$-curve. Then there
exists an effective anticanonical divisor not reduced at $P$. Thus,  $\alpha_{S_5}(P)\leqslant\frac{1}{2}$. 
Meanwhile, we have $\frac{1}{2}=\alpha(S_5)\leqslant\alpha_{S_5}(P)$ by
Lemma~\ref{lemma:inf-sup} and
Theorem~\ref{theorem:GAFA}.  Therefore,  $\alpha_{S_5}(P)=\frac{1}{2}$.

Suppose that the point $P$ is not contained in
any $(-1)$-curve.  Then there exist exactly five
irreducible smooth rational curves $C_1, \ldots, C_5$ passing through the point $P$ with  $-K_{S}\cdot C_i=2$ for each $i$
(cf. the proof of \cite[Lemma~5.8]{Ch07a}). Moreover, for every $C_i$,  there are four irreducible smooth
rational curves $E_1^i$, $E_2^i$, $E_3^i$ and $E_4^i$ such that
$3C_i+E_1^i+E_2^i+E_3^i+E_4^i$  belongs to the bi-anticanonical linear system
$|-2K_{S_5}|$ (cf.
Remark~\ref{remark:del-Pezzo-high-degree}). Therefore, 
$\alpha_{S_5}(P)\leqslant\frac{2}{3}$.

Suppose that $\alpha_{S_5}(P)<\frac{2}{3}$. Then there is an
effective anticanonical $\mathbb{Q}$-divisor $D$ such that $(S, \lambda D)$ is
not log canonical at  $P$ for some positive rational
number $\lambda<\frac{2}{3}$. Then
$\mathrm{mult}_{P}(D)>\frac{1}{\lambda}$ by
Lemma~\ref{lemma:mult}. Let $f\colon S_4\to S_5$ be the blow
up of  $S_5$ at  $P$ with the exceptional curve $E$ and  let $\tilde{D}$ be the proper transform of
the divisor $D$ on  $S_4$. Then the surface $S_4$ is a smooth del Pezzo surface of degree $4$. 
We have $$
K_{S_4}+\lambda\tilde{D}+\left(\lambda\mathrm{mult}_{P}(D)-1\right)E=f^{*}\left(K_{S_5}+\lambda D\right),%
$$
which implies that the log pair $(S_4,
\lambda\tilde{D}+(\lambda\mathrm{mult}_{P}(D)-1)E)$ is not log
canonical.

On the other hand, the log pair
$(S_4,\lambda\tilde{D}+\lambda(\mathrm{mult}_{P}(D)-1)E)$
is  log canonical because the divisor
$\tilde{D}+(\mathrm{mult}_{P}(D)-1)E$ is an effective anticanonical $\mathbb{Q}$-divisor of $S_4$ and 
$\alpha(S_4)=\frac{2}{3}$ by
Theorem~\ref{theorem:GAFA}.  However, \red{this} is absurd because $\lambda(\mathrm{mult}_{P}(D)-1)>\lambda\mathrm{mult}_{P}(D)-1$.
\end{proof}

\begin{lemma}
\label{lemma:dP4} 
The $\alpha$-function  on a del
Pezzo surface  $S_4$ of degree $4$ has the following values
$$
\alpha_{S_4}(P)=\left\{%
\aligned &2/3\ \ \ \aligned & \text{if $P$ is on a $(-1)$-curve;}
 \endaligned\\%
& 3/4\ \ \ \aligned & \aligned & \text{if there is an
effective anticanonical divisor that consists of}\\
& \text{two $0$-curves \red{meeting} tangentially at $P$;}\\ \endaligned
 \endaligned\\%
&5/6\ \ \ \text{otherwise.}\\
\endaligned\right.%
$$
\end{lemma}

\begin{proof}
Let $P$ be a point on $S_4$. 
If the point $P$ lies on a $(-1)$-curve $L$, then there are mutually disjoint five $(-1)$-curves
$E_1,\ldots,E_5$ that intersect $L$. Let $h : S_4\rightarrow
\mathbb{P}^2$ be the contraction of all $E_i$'s. Since $h(L)$ is
a conic in $\mathbb{P}^2$, we see that $3L+\sum_{1\leqslant i\leqslant
5}E_i$ is a member in the linear system $|-2K_{S_4}|$ (cf. Remark~\ref{remark:del-Pezzo-high-degree}). This means that
$\alpha_{S_4}(P)\leqslant\frac{2}{3}$. 
Therefore,
$\alpha_{S_4}(P)=\frac{2}{3}$ since
$\alpha(S_4)\leqslant\alpha_{S_4}(P)$ by Lemma~\ref{lemma:inf-sup}
and $\alpha(S_4)=\frac{2}{3}$ by Theorem~\ref{theorem:GAFA}.

Suppose that the point $P$ does not lie on a $(-1)$-curve. Put
$\omega=\frac{3}{4}$ in the case when there is an effective
anticanonical divisor that consists of two $0$-curves intersecting tangentially at  the point $P$ and put
$\omega=\frac{5}{6}$ otherwise. 

 One can easily find an effective anticanonical  divisor
$F$ on $S_4$ such that $(S_4,\lambda F)$ is not log canonical at 
$P$ for every positive rational number $\lambda>\omega$ (see \cite[Proposition~3.2]{Pa01}). This
shows that $\alpha_{S_4}(P)\leqslant\omega$.  Moreover, it is easy to check that the log pair 
 $(S_4,\omega C)$ is log canonical at 
$P$ for each $C\in|-K_{S_4}|$. 

Suppose $\alpha_{S_4}(P)<\omega$. Then there is an
effective anticanonical $\mathbb{Q}$-divisor $D$ such that $(S, \omega D)$ is
not log canonical at  $P$. 
Note that there are only finitely many effective anticanonical divisors $C_1,\ldots, C_k$ such that  each $(S_4, C_i)$ is not log canonical at  $P$.
Applying
Lemma~\ref{lemma:convexity}, we may assume that for each $i$ at least one
irreducible component of $\mathrm{Supp}(C_i)$ is not contained in the support of $D$.

Let $f\colon S_3\to S_4$ be the blow up of the surface $S_4$ at  $P$ with the exceptional curve $E$ and let
$\tilde{D}$ be the proper transform of the divisor $D$ on $S_3$.
Then $S_3$ is a smooth cubic surface in $\mathbb{P}^3$ and $E$ is a line in $S_3$. Moreover, the log pair $(S_3,
\tilde{D}+(\mathrm{mult}_{P}(D)-1)E)$ is not log canonical at some
point $Q$ on $E$ because
the log pair $(S_4,D)$ is not log canonical at   $P$.

Let $T_Q$ be the tangent hyperplane section of the cubic surface $S_3$ at  $Q$. Note that the divisor $T_Q$ contains the line $E$.
Since $\tilde{D}+(\mathrm{mult}_{P}(D)-1)E$ is an effective anticanonical $\mathbb{Q}$-divisor on $S_3$, it follows from 
Corollary~\ref{corollary:technical} 
that the log pair $(S_3, T_Q)$ is not log canonical at  $Q$ and the support of $\tilde{D}$ contains all the irreducible
components of $T_Q$. In fact, it follows that the divisor $T_Q$ is either a union of three lines meeting at  $Q$ or a union of a line and a conic intersecting tangentially at  $Q$. The divisor $f(T_Q)$ is an effective anticanonical divisor on $S_4$ such that the log pair $(S_4, f(T_Q))$ is not log canonical at  $P$. This contradicts our assumption since 
the support of $D$ contains all the irreducible components of the divisor 
$f(T_Q)$. \end{proof}

Consequently, Theorem~\ref{theorem:Park-Won} follows from
Examples~\ref{example:alpha-P^n} and \ref{example:alpha-product},
and Lemmas~\ref{lemma:F1}, \ref{lemma:dP7}, \ref{lemma:dP6},
\ref{lemma:dP5} and \ref{lemma:dP4}.

\appendix

\section{Appendix}
\label{sec:KPZ}

This appendix is devoted to the proof of Lemma~\ref{lemma:KPZ-cubic}.  The proof originates  from  \cite{KPZ11a} and \cite{KPZ12b}, where the proof is presented dispersedly. For the readers' convenience,   we
give a detailed and streamlined one here.

Let $S$ be a smooth del Pezzo surface of degree at most $4$.
Suppose that $S$ contains a $(-K_S)$-polar cylinder, i.e.,
there is an open affine subset $U\subset S$ and an effective anticanonical 
$\mathbb{Q}$-divisor $D$ such that
$U=S\setminus\mathrm{Supp}(D)$ and $U\cong Z\times \mathbb{A}^1$
for some smooth rational affine curve $Z$. Put
$D=\sum_{i=1}^{r}a_i D_i$, where each $D_i$ is an irreducible
reduced curve and each $a_i$ is a positive rational number.

\begin{lemma}[\red{{\cite[Lemma~4.4]{KPZ12b}}}]
\label{lemma:KPZ-r}
The number of the irreducible components of the divisor $D$ is not smaller than the rank of the Picard group of $S$, 
i.e.,
$r\geqslant\mathrm{rk}\,\mathrm{Pic}(S)=10-K_{S}^2\geqslant 6$.
\end{lemma}


To prove Lemma~\ref{lemma:KPZ-cubic}, we must show
that there exists a point $P\in S$ such that
\begin{itemize}
\item the log pair $(S,D)$ is not log canonical at the point $P$;%
\item if there exists a \emph{unique} divisor $T$ in the anticanonical  linear system $|-K_S|$ such that the log pair $(S, T)$ is not log canonical at $P$, then there is an effective anticanonical $\mathbb{Q}$-divisor $D'$   on the surface $S$ such that 
\begin{itemize}
\item the log pair $(S, D')$ is not log canonical at  $P$;
\item  
the support of the divisor $T$ is not contained in the support of $D'$.
\end{itemize}
\end{itemize}

The natural projection $U\cong Z\times\mathbb{A}^1\to Z$ induces a
rational map $\pi\colon S\dasharrow\mathbb{P}^1$ given by
a  pencil $\mathcal{L}$ on the surface $S$. Then either
$\mathcal{L}$ is base-point-free or its base locus consists of a
single point.

\begin{lemma}[{\cite[Lemma~4.2]{KPZ12b}}]
\label{lemma:KPZ-point} The pencil $\mathcal{L}$ is not base-point-free.
\end{lemma}

\begin{proof}
Suppose that the pencil $\mathcal{L}$ is base-point-free. Then
$\pi$ is a morphism, which implies that there exists exactly one
irreducible component of $\mathrm{Supp}(D)$ that does not lie in
a fiber  of $\pi$. Moreover, this component is a section. Without loss of generality, we may assume
that this component is $D_r$. Let $L$ be a sufficiently general
curve in  $\mathcal{L}$. Then
$$
2=-K_{S}\cdot L=D\cdot L=\sum_{i=1}^{r}a_i D_i\cdot L=a_rD_r\cdot L,
$$
and hence $a_r=2$.  \red{This implies that } $\alpha(S)\leqslant\frac{1}{2}$. However, \red{this} contradicts Theorem~
\ref{theorem:GAFA} since the degree of the surface $S$ is at most $4$.
\end{proof}

Denote the unique base point of the pencil $\mathcal{L}$  by $P$.
Let us show that  $P$ is the point we are looking for.
Resolving the base locus of the pencil
$\mathcal{L}$, we obtain a commutative diagram
$$
\xymatrix{ &W\ar[dl]_f\ar[dr]^g& \\
S\ar@{-->}[rr]^{\pi}&&\mathbb{P}^1,}
$$
where $f$ is a composition of blow ups at smooth points over  $P$ and $g$
is a morphism whose general fiber is a smooth rational curve.
Denote by $E_1, \ldots, E_n$ the exceptional curves of the
birational morphism $f$. Then there exists exactly one curve among
them that does not lie in the fibers of the morphism $g$. Without loss of generality, we may assume that
this curve is $E_n$. Then $E_n$ is a section of the morphism $g$.

For every $D_i$, denote by $\tilde{D}_i$ its proper transform on
the surface $W$. Then every curve $\tilde{D}_i$ lies in a fiber
of the morphism $g$. 

\red{The following lemma is a bit  stronger version of \cite[Lemma~4.6]{KPZ12b} even though its proof is almost the same as that of  \cite[Lemma~4.6]{KPZ12b}. }
\begin{lemma}[\red{{cf. \cite[Lemma~4.6]{KPZ12b}}}]
\label{lemma:KPZ-special} For every effective anticanonical 
$\mathbb{Q}$-divisor $H$ with 
$\mathrm{Supp}(H)\subseteq\mathrm{Supp}(D)$,  the
log pair $(S,H)$ is not log canonical at the point $P$.
\end{lemma}

\begin{proof}
\red{The proof of  \cite[Lemma~4.6]{KPZ12b} works verbatim for this generalized version.}
\end{proof}

Applying Lemma~\ref{lemma:KPZ-special} to $(S,D)$, we see that the
log pair $(S,D)$ is not log canonical at $P$. Thus, if there
exists no anticanonical divisor $T$ such that $(S,T)$ is not log
canonical at $P$, then we are done. Hence, to complete the proof
of Lemma~\ref{lemma:KPZ-cubic}, we assume that
there exists a \emph{unique} divisor $T\in |-K_{S}|$ such that
$(S,T)$ is not log canonical at $P$. Then
Lemma~\ref{lemma:KPZ-cubic} follows from the lemma below.

\begin{lemma}
\label{lemma:KPZ-final} There exists an effective \red{anticanonical} 
$\mathbb{Q}$-divisor $D^{\prime}$ on $S$ such that
 the log pair $(S,D^{\prime})$ is
not log canonical at  $P$ and
$\mathrm{Supp}(D^{\prime})$ does not contain at least one
irreducible component of $\mathrm{Supp}(T)$.
\end{lemma}

\begin{proof}
If $\mathrm{Supp}(D)$ does not contain at least one irreducible
component of $\mathrm{Supp}(T)$, then we can simply put
$D=D^\prime$. Suppose that it is not the case, i.e., we
have $\mathrm{Supp}(T)\subseteq\mathrm{Supp}(D)$. Then $T\ne D$.
Indeed, the number of the irreducible components of $\mathrm{Supp}(D)$
is at least $6$ by Lemma~\ref{lemma:KPZ-r}. On the other hand, the
number of the irreducible components of $\mathrm{Supp}(T)$ is at most
$4$ because $-K_{S}\cdot T=K_{S}^2$ and $-K_{S}$
is ample.

Since $T\ne D$, there exists a positive rational number $\mu$ such
that the $\mathbb{Q}$-divisor $(1+\mu)D-\mu T$ is effective and
its support does not contain at least one irreducible component of
$\mathrm{Supp}(T)$. Put
$D^{\prime}=(1+\mu)D-\mu T$.  Note that $D'$ is also an effective anticanonical $\mathbb{Q}$-divisor on $S$.
By our construction, $\mathrm{Supp}(D^\prime)\subseteq\mathrm{Supp}(D)$.
Thus, the log pair $(S,D')$ is
not log canonical at  $P$ by Lemma~\ref{lemma:KPZ-special}. This completes
the proof.
\end{proof}

\begin{remark}
Note that $U\ne S\setminus\mathrm{Supp}(D^\prime)$, which implies
that the number of the irreducible components of
$\mathrm{Supp}(D^\prime)$ may be less than
$\mathrm{rk}\,\mathrm{Pic}(S)$. Because of this, we can apply
Lemma~\ref{lemma:convexity} only once here. This shows that we
really need to use the \emph{uniqueness} of the divisor $T$ in the anticanonical linear system
$|-K_{S}|$ such that $(S,T)$ is not log canonical at $P$  in the
proof of Lemma~\ref{lemma:KPZ-final}. Indeed,
if there is another divisor $T^\prime$ in $|-K_{S}|$ such that $(S,T^\prime)$ is  not log canonical at $P$,
then we would not be able to apply Lemma~\ref{lemma:convexity}
since we \red{might} have $D^{\prime}=T^\prime$.
\end{remark}

\bigskip
\footnotesize
\noindent\textit{Acknowledgments.}
The authors would like to express their thanks to the referee for the careful reading and the comments that improve this article. 
The first author was supported by AG Laboratory NRI-HSE, RF government grant, 11.G34.31.0023. The second author was supported by the Institute for Basic Science  (Grant No. CA1305-02).

\end{document}